\documentclass[11pt]{amsart}
\usepackage{wasysym,amssymb,eufrak,indentfirst,graphicx,cite,amsthm,color}
\usepackage[bookmarksnumbered, colorlinks, linktocpage, plainpages]{hyperref}
\usepackage{epsfig}
\usepackage{float}
\usepackage{indentfirst}
\newtheorem{theorem}{Theorem}[section]
\newtheorem{corollary}[theorem]{Corollary}
\newtheorem{lemma}[theorem]{Lemma}
\newtheorem{proposition}[theorem]{Proposition}
\theoremstyle{definition}
\newtheorem{definition}[theorem]{Definition}

\theoremstyle{remark}
\newtheorem{remark}[theorem]{Remark}
\newtheorem{example}[theorem]{Example}
\numberwithin{equation}{section}
\begin{document}

\title[On geometric aspects of  slice regular functions]
{On geometric aspects of quaternionic and octonionic slice regular functions}

\author[X. P. Wang]{Xieping Wang}
\address{Xieping Wang, Department of Mathematics, University of Science and
Technology of China, Hefei 230026, China}

\email{pwx$\symbol{64}$mail.ustc.edu.cn}



\begin{abstract}
The purpose  of this paper is twofold. One is to enrich from a geometrical point of view the theory of octonionic slice regular functions. We first prove a boundary Schwarz lemma for slice regular self-mappings of the open unit ball  of the octonionic space. As applications, we obtain two Landau-Toeplitz type theorems for slice regular functions with respect to regular diameter and slice diameter respectively, together with a Cauchy type estimate. Along with these results, we introduce some  new and useful ideas, which also allow to prove the minimum principle and one version of the open mapping theorem. Another is to  strengthen a version of boundary Schwarz lemma first proved in \cite{WR} for quaternionic slice regular functions, with a completely new approach. Our quaternionic boundary Schwarz lemma with optimal estimate  improves considerably a well-known Osserman type estimate  and provides additionally  all the extremal functions.

\end{abstract}

\maketitle

\tableofcontents
\section{Introduction}
A promising theory of quaternion-valued functions of one quaternionic variable, now called slice regular functions, was initially introduced by Gentili and Struppa in \cite{GS1, GS2} and has been extensively studied over the past few years. It turns out to be significantly different from the more classical theory of regular (or monogenic) functions in the sense of Cauchy-Fueter and has  elegant applications to the functional calculus for noncommutative operators \cite{Co2}, Schur analysis \cite{ACS} and the construction and classification of orthogonal complex structures on dense open subsets of $\mathbb R^4$ \cite{GSS2014}.  Meanwhile, the theory of quaternionic slice regular functions has been extended to octonions in \cite{GS50}. The related theory of slice monogenic functions on domains in the paravector space $\mathbb R^{n+1}$ with values in the Clifford algebra $\mathbb R_n$ was introduced in \cite{Co6,Co3}. For the detailed up-to-date theory, we refer the reader to the monographs \cite{GSS, Co2} and the references therein. More recently, a connection between slice monogenic
and   monogenic functions was investigated in \cite{CLSS} by means of  Radon and dual Radon transforms. These function theories were also unified and generalized in \cite{Ghiloni1} by means of the concept of slice functions on the so-called quadratic cone of a real alternative *-algebra, based on a  slight modification of a well-known construction due to Fueter. The theory of slice regular functions on  real alternative *-algebras  is by now well-developed through a series of papers \cite{Ghiloni5, Ghiloni3,Ghiloni4,Ghiloni6,Ghiloni7} mainly due to Ghiloni and Perotti after their seminal work \cite{Ghiloni1}.

Exactly as the quaternions $\mathbb H$ being the only real \textit{associative normed division } algebra of dimension greater than 2, the theory of quaternionic slice regular functions should be the most beautiful one  among these function theories mentioned above. This is indeed the case. Such a special class of functions enjoys many nice properties similar to those of classical holomorphic functions of one complex variable.
Among them, we particularly mention the open mapping theorem, which is by now  known to hold only for slice regular functions on symmetric slice domains in $\mathbb H$ and allows us to prove the Koebe type one-quarter theorem for slice regular extensions to the  quaternionic ball of  univalent holomorphic functions on the unit disc of the complex plane, see Theorem \cite[Theorem 4.9]{RW2} for details. Furthermore, from  the analytical perspective, only for quaternionic slice regular functions, the regular product and regular quotient have  an intimate connection with the usual pointwise product and quotient. It is exactly this connection which plays a crucial role in many arguments, see the monograph \cite{GSS} and the recent paper \cite{WR} for more details.

Now a rather natural question arises of whether the nice properties enjoyed by quaternionic slice regular functions can be proved for  slice regular functions on domains over the octonions $\mathbb O$,  the only   \textit{normed division} algebra among alternative algebras over $\mathbb R$ of dimension greater than 4. In this paper, we introduce some  new and useful ideas to overcome difficulties brought by the non-commutativity and the non-associativity and in turn to show that to great extent, this is indeed the case.

First of all, we are going to   focus on the boundary behavior of octonionic slice regular functions, analogous to that of holomorphic functions. More specifically, we shall prove a boundary Schwarz lemma for slice regular self-mappings of the open unit ball $\mathbb B:=\{w\in \mathbb O: |w|<1\}$. To state its precise content, we first introduce some necessary notations. For a given element $\xi=x+y I\in\mathbb O$ with
$I$ being an element of the unit 6-dimensional of purely imaginary octonions
\begin{equation}\label{Pure-sphere}
\mathbb S=\big\{w \in \mathbb O:w^2 =-1\big\},
\end{equation}
we denote by $\mathbb S_{\xi}$ the associated 6-dimensional sphere (reduces to the point $\xi$ when $\xi$ is real):
$$\mathbb S_{\xi}:=x+y\, \mathbb S=\big\{x+y J: J\in\mathbb S\big\}.$$
It is well known  that $\mathbb S_{\xi}$ is exactly the conjugacy class of $\xi$ (cf. \cite[Proposition 2, Corollary 2.1]{Serodio}). For any three octonions $u, v, w\in\mathbb O$, the \textit{Lie bracket} of $u,v $ and the \textit{associator} of  $u, v, w$ are respectively defined to be
$$[u, v]:=uv-vu,\qquad [u, v, w]:=(uv)w-u(vw).$$
We also denote by $\langle$ , $\rangle$ the standard Euclidean inner product on $\mathbb O\cong\mathbb R^8$. Now our first main result can be stated as follows:

\begin{theorem} \label{BSL}
Let $\xi\in \partial\mathbb B$ and $f$ be a slice regular function on $\mathbb B\cup\mathbb S_{\xi}$ such that $f(\mathbb B)\subseteq\mathbb B$ and $f(\xi)\in \partial\mathbb B$. Then
\begin{enumerate}
  \item [(i)] it holds that
  \begin{equation}\label{Schwarz ineq1}
\begin{split}
\frac{\partial |f|}{\partial \xi}(\xi)&
=\overline{\xi}\Big(f(\xi)\overline{f'(\xi)}+\big[\bar{\xi}, f(\xi)\overline{R_{\bar{\xi}}R_{\xi}f(\xi)}\,\big]+2\big[\xi, f(\xi), R_{\bar{\xi}}R_{\xi}f(\xi)\big]\Big)
\\
&\geq\frac{\big|1-\big\langle f(0), f(\xi)\big\rangle\big|^2}{1-|f(0)|^2},
\end{split}
\end{equation}
where $\frac{\partial |f|}{\partial \xi}(\xi)$ is the directional derivative of $|f|$ along the direction $\xi$  at the boundary point $\xi\in \partial\mathbb B$;

  \item [(ii)] if further  $f(0)=0$ and $f(\xi)=\xi$, then
\begin{equation}\label{Schwarz ineq2}
\frac{\partial |f|}{\partial \xi}(\xi)= f'(\xi)-\big[\xi, R_{\bar{\xi}}R_{\xi}f(\xi)\,\big] \geq\frac{2}{1+{\rm{Re}}f'(0)}.
\end{equation}
Moreover, equality holds for  the inequality in {\rm{(\ref{Schwarz ineq2})}} if and only if $f$ is of the form
\begin{equation}\label{f-expression}
f(w)=w\big(1-wa\bar{\xi}\,\big)^{-\ast}\ast\big(w\bar{\xi}-a\big)
\end{equation}
for some constant $a\in [-1,1)$.
\end{enumerate}
\end{theorem}

For the precise definitions of $R_{\bar{\xi}}R_{\xi}f(\xi)$ and $\ast$-product appeared in Theorem \ref{BSL}, see (\ref{de:Rf1}), (\ref{de:Rf2}) and Sect. 2 below. It turns out that  $R_{\bar{\xi}}R_{\xi}f(\xi)$ is intimately related to  the second coefficient in a new series expansion of slice regular function $f$, see \cite[Theorem 4.1]{Stop3} for the quaternionic case and \cite[Theorem 5.4]{Ghiloni3} for the real alternative *-algebra case.
Moreover, it is well worth remarking here that in contrast to the complex case, the Lie bracket in (\ref{Schwarz ineq2})  does  not necessarily vanish and $f'(\xi)$ may not be a real number. An explicit counterexample can be found in Example \ref{Non-real-example} below.

Although the key ingredient in proving Theorem $\ref{BSL}$ is still a careful consideration of the geometrical information of $f$ at its prescribed contact point $\xi$ (i.e. $f(\xi)\in \partial\mathbb B$), two crucial difficulties arise in the octonionic setting. One is that the case of $\xi$ being a contact point of $f$ (i.e. $f(\xi)\in\partial\mathbb B$) can not be reduced to the case of $\xi$ being a boundary fixed point of $f$ (i.e. $f(\xi)=\xi\in\partial\mathbb B$); the other is that, because of the lack of associativity in  $\mathbb O$, there is in general no nice connection between the regular product and the usual pointwise product unlike in the quaternionic setting. The peculiarities of the non-associative setting produce a completely new phenomenon, called the \textit{camshaft effect} in \cite{Ghiloni2}: an isolated zero of a slice regular function $f$ is not necessarily a zero for the regular product $f\ast g$ of $f$ with another slice  regular function $g$. Therefore, the method used in our recent work \cite{WR} fails in the present setting to get some satisfactory and even sharp estimate. Fortunately, we can come up with an effective approach to overcome partly these technical difficulties mentioned above. In the quaternionic case, with a completely new approach, we can   strengthen a result first proved in \cite{WR} by the author and Ren, analogous to Theorem \ref{BSL}. Our quaternionic boundary Schwarz lemma with optimal estimate involves a Lie bracket, improves considerably a well-known Osserman type estimate  and provides additionally  all the extremal functions; see Theorem \ref{Generalized Herzig} below for details.

Let $f$ be as described in Theorem $\ref{BSL}$. Notice that   the directional derivative $\frac{\partial f}{\partial \xi}(\xi)$ of $f$ along the direction $\xi$  at the boundary point $\xi\in \partial\mathbb B$ satisfies that
$$\frac{\partial f}{\partial \xi}(\xi)=\xi f'(\xi),$$
thus the obvious inequality
$$\Big|\frac{\partial f}{\partial \xi}(\xi)\Big|\geq \frac{\partial |f|}{\partial \xi}(\xi)$$ results in:

\begin{corollary}
Let $\xi\in \partial\mathbb B$ and $f$ be a slice regular function on $\mathbb B\cup\{\xi\}$ such that $f(\mathbb B)\subseteq\mathbb B$, $f(0)=0$ and $f(\xi)=\xi$. Then
\begin{equation*}
|f'(\xi)|\geq\frac{2}{1+{\rm{Re}}f'(0)}.
\end{equation*}
Moreover, equality holds for the last  inequality if and only if $f$ is of the form
\begin{equation*}
f(w)=w\big(1-wa\bar{\xi}\,\big)^{-\ast}\ast\big(w\bar{\xi}-a\big)
\end{equation*}
for some constant $a\in [-1,1)$.
\end{corollary}

We shall give some applications of Theorem \ref{BSL} to the study of geometric properties and rigidity of slice regular functions. We first recall the notion of \textit{regular diameter}, a suitable tool to measure the image of the open unit ball $\mathbb B$ of the octonionic space $\mathbb O$ through a slice  regular function.

\begin{definition}
Let $f$ be a  slice  regular function on $\mathbb B$ with Taylor expansion
$$f(w)=\sum\limits_{n=0}^{\infty}w^na_n.$$ For each $r\in(0, 1)$, the \textit{regular diameter} of the image of $r\mathbb B$ under $f$ is defined to be
\begin{equation}\label{regular-diam}
\widetilde{d}\big(f(r\mathbb B)\big):=\max_{u,v\in \overline{\mathbb B}}\max_{|w|\leq r}|f_u(w)-f_v(w)|,
\end{equation}
where
$$f_u(w):=\sum\limits_{n=0}^{\infty}w^n(u^na_n), \qquad f_v(w):=\sum\limits_{n=0}^{\infty}w^n(v^na_n).$$
The \textit{regular diameter} of the image of $\mathbb B$ under $f$ is defined to be
\begin{equation}\label{regular-diam}
\widetilde{d}\big(f(\mathbb B)\big):=\lim_{r\rightarrow 1^-}\widetilde{d}\big(f(r\mathbb B)\big).
\end{equation}
\end{definition}

As a first application of Theorem \ref{BSL}, we have the following Landau-Toeplitz type theorem  for octonionic slice  regular functions, whose quaternionic version was proved in \cite{Gen-Sar}.
\begin{theorem}\label{regular diam-LT}
Let $f$ be a slice  regular function on $\mathbb B$ such that
$$\widetilde{d}\big(f(\mathbb B)\big)=2.$$
Then
\begin{equation}\label{regular-diam01}
\widetilde{d}\big(f(r\mathbb B)\big)\leq 2r
\end{equation}
for each $r\in(0, 1)$, and
\begin{equation}\label{regular-diam02}
 |f'(0)|\leq 1.
\end{equation}
Moreover, equality holds in $(\ref{regular-diam01})$ for some $r_0\in(0,1)$, or in $(\ref{regular-diam02})$, if and only if $f$ is an affine function
$$f(w)=f(0)+wf'(0).$$
\end{theorem}

Let $E, \Omega$ be two subsets of $\mathbb O$ and $f: \Omega\rightarrow \mathbb O$ a function. We denote by ${\rm{diam}}\, E=\sup_{z, w\in E}|z-w|$ the Euclidean diameter of $E$ and define the \textit{slice diameter} of the image of $\Omega$ under $f$ to be
\begin{equation}\label{slice-diam}
\widehat{d}\big(f(\Omega)\big):=\sup_{I\in\mathbb S}{\rm{diam}}\, f(\Omega_I),
\end{equation}
where $\Omega_I$ denotes the intersection $\Omega\cap\mathbb C_I$ of $\Omega$ and $\mathbb C_I$ the complex plane determined by $I$, and $\mathbb S$ is the same as in (\ref{Pure-sphere}).
Thus we have another version of Landau-Toeplitz type theorem with respect to slice diameter.
\begin{theorem}\label{slice diam-LT}
Let $f$ be a slice  regular function on $\mathbb B$ such that
$$\widehat{d}\big(f(\mathbb B)\big)=2.$$
Then
\begin{equation}\label{slice-diam01}
{\rm{diam}}\,\big(f(r\mathbb B_I)\big)\leq 2r
\end{equation}
for each $r\in(0, 1)$ and each $I\in\mathbb S$, and
\begin{equation}\label{slice-diam02}
 |f'(0)|\leq 1.
\end{equation}
Moreover, equality holds in $(\ref{slice-diam01})$ for some $r_0\in(0,1)$ and $I_0\in\mathbb S$, or in $(\ref{slice-diam02})$, if and only if $f$ is an affine function
$$f(w)=f(0)+wf'(0).$$
\end{theorem}

As a second application of Theorem \ref{BSL}, we have the following Cauchy type estimate, which is an analogue of an old result due to Poukka (see \cite{Poukka}):

\begin{theorem}\label{Poukka}
Let $f$ be a bounded  slice regular function on $\mathbb B$ and $d:={\rm{Diam}}\, f(\mathbb B)$ the Euclidean diameter of the image set $f(\mathbb B)$. Then the inequality
\begin{equation}\label{Cauchy type}
\frac{|f^{(n)}(0)|}{n!}\leq \frac12d
\end{equation}
holds for every positive  integer $n\in \mathbb N$. Moreover, equality holds in $(\ref{Cauchy type})$ for some $n_0\in \mathbb N$ if and only if
$$f(w)=f(0)+\frac12 w^{n_0}d\,e^{I\theta}$$
for some $I\in\mathbb S$ and some $\theta \in \mathbb R$.
\end{theorem}
It is noteworthy here that inequality $(\ref{Cauchy type})$ easily follows from the classical result due to Poukka together with   the splitting lemma for slice regular functions, or alternatively from Cauchy integral formula. The point here is to prove the last statement in the theorem.

Next we use some ideas developed in the proof of Theorem \ref{BSL} to prove other properties of octonionic slice regular functions, among which are the minimum principle and the open mapping theorem.

\begin{theorem}\label{W-Min-Principle}
Let $f:\Omega\rightarrow \mathbb O$ be a slice regular function on a symmetric slice domain $\Omega\subseteq\mathbb O$. If  $|f|$  attains a local minimum at some point $w_0\in \Omega\cap\mathbb R$, then either $f(w_0)=0$ or $f$ is constant.
\end{theorem}

We shall give two proofs of the preceding theorem. Both of them involve a variational argument. The first one also provides a completely new and quite elementary approach to the maximum and the minimum principles for holomorphic functions of one complex variable. The second one is to reduce this theorem to the maximum   principle (Theorem \ref{MP}), based on a nice connection between the Euclidean norm of the slice regular function $f$ and that of its regular reciprocal $f^{-\ast}$ (Proposition \ref{R-product wrt I-product}). Furthermore, it seems that the restriction of $w_0$ belonging to $\Omega\cap\mathbb R$ in the preceding theorem is superfluous. There are some additional obstacles to prove  the general case that $w_0\in\Omega\setminus\mathbb R$; see Remark \ref{Remark on WMP} below for more details. If this  restriction could be  removed, the general minimum principle would immediately follow and in turn would imply the open mapping theorem analogous to \cite[Theorem 7.7]{GSS}. Here we can merely prove the following version of the open mapping theorem using a method different from that of \cite[Theorem 7.4]{GSS}.

\begin{theorem}\label{Weak Open Mapping Theorem}
Let $f:\Omega\rightarrow \mathbb O$ be a nonconstant slice regular function on a symmetric slice domain $\Omega\subseteq \mathbb O$. If $U$ is a symmetric open subset of $\Omega$, then $f(U)$ is open. In particular, $f(\Omega)$ is open.
\end{theorem}

Theorem \ref{Weak Open Mapping Theorem} is sufficient for proving an octonionic version of the classical Koebe one-quarter theorem for  slice regular extensions to the  octonionic ball $\mathbb B$ of  univalent holomorphic functions on the unit disc of the complex plane.

\begin{theorem} \label{th:Koebe-theorem}
Let $f$ be a slice regular function on  $\mathbb B$ such that its restriction $f_I$ to $\mathbb B_I$ is injective and $f(\mathbb B_I)\subseteq \mathbb C_I $ for some $I\in \mathbb S $. If $f(0)=0$ and $f'(0)=1$, then it holds that
$$B(0,\frac14)\subset f(\mathbb B).$$
\end{theorem}

The remaining part of this paper is organized as follows. In Sect. \ref{Preliminaries}, we set up basic notations and give some preliminary results from the theory of octonionic slice regular functions. In Sect. \ref{Proof of Theorem BSL}, we first establish some useful lemmas and then use them to prove Theorem \ref{BSL}. Sect. \ref{applications of Thm BSL} is devoted to the detailed proofs of Theorems \ref{regular diam-LT}, \ref{slice diam-LT} and \ref{Poukka}. In Sect. \ref{Geometric properties}, we first use some ideas developed in the proof of Theorem \ref{BSL} to prove   Theorems \ref{W-Min-Principle} and \ref{Weak Open Mapping Theorem}. We then use Theorem \ref{Weak Open Mapping Theorem} and a new convex combination identity (Proposition \ref{convex combination identity}) to prove the growth and distortion theorems (Theorem \ref{Growth and Distortion Theorems}) and Theorem \ref{th:Koebe-theorem}. Finally, in Sect. \ref{Quaternionic BSL} we use Julia lemma in \cite{WR} to prove a new and sharp boundary Schwarz lemma for quaternionic slice regular self-mappings of the open unit ball of the quaternions (Theorem \ref{Generalized Herzig}) and give some consequences. This paper is closed with a comparison of these results and the corresponding results for holomorphic self-mappings of the open unit disc on the complex plane.

\section{Preliminaries}\label{Preliminaries}
We recall in this section some necessary definitions and preliminary results on octonions and octonionic slice regular functions that we need later on.

\subsection{Octonions}
We denote by  $\mathbb O$ the non-commutative and non-associative division algebra of octonions (also called Cayley numbers). We refer to \cite{Lam,Okubo,Schafer} for a more complete insight on octonions; here we shall just recall what is need for our purpose. A simple way to describe its construction is to consider a basis
$\mathcal{E}=\{e_0=1, e_1,\ldots, e_6, e_7\}$ of $\mathbb R^8$ and relations
\begin{equation}\label{Generation rule01}
e_ie_j=-\delta_{ij}+\psi_{ijk}e_k, \quad i, j, k=1,2,\ldots,7,
\end{equation}
where $\delta_{ij}$ is the Kronecker delta, and $\psi_{ijk}$ is \textit{totally antisymmetric} in $i, j, k$, non-zero and equal to one for the seven combinations in the following set
$$\Sigma=\big\{(1, 2, 3), (1, 4, 5), (2, 4, 6), (3, 4, 7), (5, 3, 6), (6, 1, 7), (7, 2, 5)\big\}$$
so that every element in $\mathbb O$ can be uniquely written as $w=x_0+\sum_{k=1}^7x_ke_k$, with $x_k (k=1, 2, 3, 4)$ being real numbers. The full multiplication table is conveniently encoded in a 7-point projective plane, the so-called Fano mnemonic graph, shown in  Fig. 1 below. In the Fano mnemonic graph, the vertices are labeled by $1, \ldots, 7$
instead of $e_1, \ldots, e_7$. Each of the 7 oriented lines gives a quaternionic triple. The
product of any two imaginary units is given by the third unit on the unique line
connecting them, with the sign determined by the relative orientation.

\begin{figure}[H]\label{figure}
\centering\includegraphics[width=4cm]
{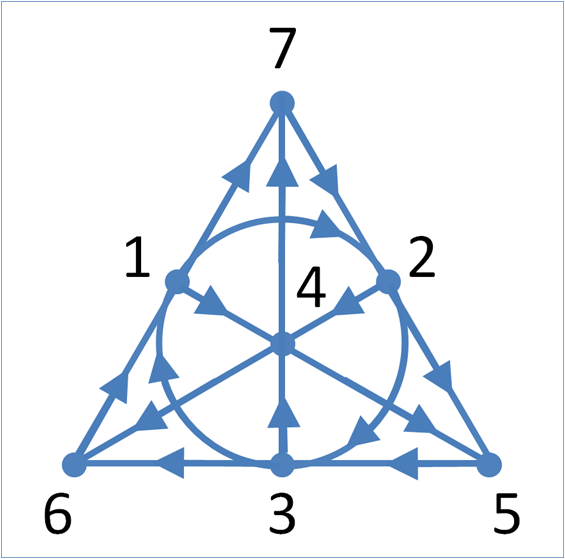}
\caption{Fano Mnemonic}
\end{figure}

Alternatively, $\mathbb O$ can be obtained from the quaternions $\mathbb H$ by the well-known \textit{Cayley-Dickson process}, which goes as follows. Let $\{1, e_1, e_2, e_3:=e_1e_2\}$ denote a real basis of $\mathbb H$. Each element $w\in\mathbb O$ can be written as $w=w_1+w_2e_4$, where $w_1, w_2\in\mathbb H$ and $e_4$ is a fixed imaginary unit of $\mathbb O$. The addition on $\mathbb O$ is defined componentwisely and the product is defined by
\begin{equation}\label{Generation rule02}
zw=(z_1+z_2e_4)(w_1+w_2e_4):=z_1w_1-\overline{w}_2z_2+(z_2\overline{w}_1+w_2z_1)e_4
\end{equation}
for all $z=z_1+z_2e_4$, $w=w_1+w_2e_4\in\mathbb O$, where $\overline{w}_1$, $\overline{w}_2$ are the conjugates of the quaternions $w_1, w_2\in\mathbb H$. Set $e_5:=e_1e_4$, $e_6:=e_2e_4$, $e_7:=e_3e_4=(e_1e_2)e_4$. Then $\{1, e_1, e_2,\ldots, e_7\}$ forms a real basis of $\mathbb O$, and one can easily verify that the product rule given by (\ref{Generation rule02}) is the same as the one in (\ref{Generation rule01}), and hence these two approaches indeed yield the same algebra $\mathbb O$.

For each $w=x_0+\sum_{k=1}^7x_ke_k\in\mathbb O$, the real number $x_0$ is called the \textit{real part} of $w$, and is denoted by ${\rm{Re}}(w)$, while $\sum_{k=1}^7x_ke_k$ is called the \textit{imaginary part} of $w$  and is denoted by ${\rm{Im}}(w)$. Moreover, we can define in a natural fashion the \textit{conjugate} $\overline{w}:=x_0-\sum_{k=1}^7x_ke_k\in\mathbb O$, and the \textit{squared norm} $|w|^2:=w\overline{w}=\overline{w}w=\sum_{k=0}^7x_k^2$ (and by the Artin's theorem below, $|zw|=|z||w|$ for any $z,w\in\mathbb O$), which is induced by the standard Euclidean inner product on $\mathbb O\cong \mathbb R^8$ given by
\begin{equation}\label{inner product on O}
\langle z, w\rangle=\textrm{Re}(z\overline{w})=\frac12(z\overline{w}+w\overline{z}), \qquad \forall\, z,w\in\mathbb O.
\end{equation}
Also,
\begin{equation}\label{inner and norm on O}
\langle z, w\rangle=\frac12\big(|z+w|^2-|z|^2-|w|^2\big), \qquad \forall\, z,w\in\mathbb O.
\end{equation}
The \textit{associator} of three octonions $u, v, w\in\mathbb O$ is defined to be
$$[u, v, w]:=(uv)w-u(vw),$$
which is \textit{totally antisymmetric}
in its arguments $u, v, w\in\mathbb O$ and has \textit{no real part}, i.e.
\begin{equation}\label{Real part free}
{\rm{Re}}\,[u, v, w]=0.
\end{equation}
Although the associator does not vanish in
general, the octonions do satisfy a weak form of associativity known as \textit{alternativity},
namely the so-called Moufang identities (cf. \cite[p. 120]{Harvey}; also \cite[p. 18]{Okubo}):
\begin{equation}\label{Monfang}
(uvu)w=u(v(uw)),\quad w(uvu)=((wu)v)u, \quad u(vw)u=(uv)(wu).
\end{equation}
The underlying reason for this is the so-called \textit{Artin's theorem}, which can be stated as follows.
\begin{theorem}{\rm{(cf. \cite[p. 18]{Okubo})}}\label{Artin-thm}
In an alternative algebra $A$, every  subalgebra generated by any two elements of $A$ is always associative.
\end{theorem}

For each $\alpha\in \mathbb O$ with $|\alpha|=1$, we now consider two multipliers $\mathcal{L}_{\alpha}$ and $\mathcal{R}_{\alpha}$ on the octonionic space $\big(\mathbb O, \langle \,,\,\rangle\big)$ associated with $\alpha$, induced respectively by left and right multiplications, i.e.
$$\mathcal{L}_{\alpha}(w)=\alpha w,\qquad \mathcal{R}_{\alpha}(w)= w\alpha, \qquad \forall\, w\in\mathbb O.$$
Clearly, $\mathcal{L}_{\alpha}$ and $\mathcal{R}_{\alpha}$ are two $\mathbb R$-linear bijections with inverses $\mathcal{L}_{\alpha^{-1}}$ and $\mathcal{R}_{\alpha^{-1}}$, respectively. Moreover, they are two unitary operators on $\big(\mathbb O, \langle \,,\,\rangle\big)$ in virtue of equality $(\ref{inner and norm on O})$. Therefore, we have the following simple lemma.

\begin{lemma}\label{unitary multipliers}
For each $\alpha\in \mathbb O$ with $|\alpha|=1$, $\mathcal{L}_{\alpha}$ and $\mathcal{R}_{\alpha}$ are two unitary operators on the octonionic space $\big(\mathbb O, \langle \,,\,\rangle\big)$.
\end{lemma}

As a direct consequence of the preceding lemma, we have the following result.

\begin{lemma}\label{associator}
For any three octonions $u, v, w\in \mathbb O$, it holds that
\begin{equation}\label{associator01}
 \big\langle u,\, [u, v, w]\big\rangle=0.
\end{equation}
\end{lemma}

\begin{proof}
First, we prove that
$$ \big\langle I,\, [I, v, w]\big\rangle=0$$
for every $I\in\mathbb S$, $v, w\in \mathbb O$. Indeed,
\begin{equation*}
\begin{split}
\big\langle I,\, [I, v, w]\big\rangle
&=\big\langle I,\, (Iv)w\big\rangle-\big\langle I,\,  I(vw)\big\rangle
\\
&=-\big\langle 1, \big((Iv)w\big)I \big\rangle
-\big\langle 1, vw\big\rangle \ \ \,\quad \qquad \qquad \mbox{by Lemma \ref{unitary multipliers}}
\\
&=-\big\langle 1, [Iv, w, I]+ (Iv)(w I) \big\rangle
-\big\langle 1, vw\big\rangle
\\
&=-\big\langle 1, (Iv)(wI) \big\rangle
-\big\langle 1, vw\big\rangle \ \ \   \quad \qquad \qquad \mbox{by (\ref{Real part free})}
\\
&=-\big\langle 1, I(vw) I \big\rangle
-\big\langle 1, vw\big\rangle \ \ \, \qquad \qquad \qquad \mbox{by (\ref{Monfang})}
\\
&=\big\langle 1, vw\big\rangle
-\big\langle 1, vw\big\rangle \ \,\quad\qquad\qquad \qquad \qquad \mbox{by Lemma \ref{unitary multipliers}}
\\
&=0.
\end{split}
\end{equation*}
For each $u\in\mathbb O$. We write $u=x+yI$ with $x, y\in\mathbb R$ and $I\in\mathbb S$, then
\begin{equation*}
\begin{split}
\big\langle u,\, [u, v, w]\big\rangle
&= \big\langle u,\, [yI, v, w]\big\rangle
\\
&= y\big\langle x+yI, \, [yI, v, w]\big\rangle
\\
&= y^2\big\langle I,\, [I, v, w]\big\rangle
\\
&=0,
\end{split}
\end{equation*}
which completes the proof.
\end{proof}

Despite the triviality of  the above two lemmas, they turn out to be quite useful in our subsequent argument, especially in the proofs of Theorems \ref{BSL}, \ref{regular diam-LT} and Proposition \ref{R-product wrt I-product}.

\subsection{Octonionic slice regular functions}
In order to introduce the notion of slice regularity on octonions $\mathbb O$, we rewrite
each element $w\in \mathbb O$ as $w = x + yI$, where $x, y \in \mathbb R$ and
$$I=\dfrac{{\rm{Im}}\, (w)}{|{\rm{Im}}\, (w)|}$$
 if ${\rm{Im}}(w)\neq 0$, otherwise we take $I$ arbitrarily such that $I^2=-1$.
Then $I $ is an element of the unit 6-dimensional sphere of purely imaginary octonions
$$\mathbb S=\big\{w \in \mathbb O:w^2 =-1\big\}.$$
For any two elements $I, J\in\mathbb S$, we define the \textit{wedge product} of $I$ and $J$ as
$$I\wedge J:=\frac12[I, J]=\frac12(IJ-JI),$$
which satisfies that
\begin{equation}\label{relation-inner-wedge}
IJ=-\langle I,J\rangle+I\wedge J,
\end{equation}
in view of (\ref{inner product on O}).
For every $I \in \mathbb S $ we will denote by $\mathbb C_I$ the plane $ \mathbb R \oplus I\mathbb R $, isomorphic to $ \mathbb C$, and, if $\Omega \subseteq \mathbb O$, by $\Omega_I$ the intersection $ \Omega \cap \mathbb C_I $. Also, we will denote by $B(p, R)$ the Euclidean open  ball of radius $R$ centred at $p\in\mathbb O$, i.e.
$$B(p, R)=\big\{w \in \mathbb O:|w-p|<R\big\}.$$
For simplicity, we denote by $\mathbb B$ the ball $B(0, 1)$.

We can now recall the notion of slice regularity.

\begin{definition} \label{de: regular} Let $\Omega$ be a domain in $\mathbb O$. A function $f :\Omega \rightarrow \mathbb O$ is called (left) \emph{slice regular} if, for all $ I \in \mathbb S$, its restriction $f_I$ to $\Omega_I$ is \emph{holomorphic}, i.e., it has continuous partial derivatives and satisfies
$$\bar{\partial}_I f(x+yI):=\frac{1}{2}\left(\frac{\partial}{\partial x}+I\frac{\partial}{\partial y}\right)f_I (x+yI)=0$$
for all $x+yI\in \Omega_I $.
\end{definition}

A wide class of examples of slice regular functions is given by polynomials and  power series of the variable $w$ with octonionic coefficients on the right. Indeed, a function $f$ is slice regular on an  open   ball $B(0,R)$ if and only if $f$ admits a power series expansion
\begin{equation}\label{Taylor expansion on ball}
f(w)=\sum_{n=0}^{\infty}w^na_n,
\end{equation}
which converges absolutely and uniformly on every compact subset of $B(0,R)$ (see \cite{GS50}).
As shown in \cite{CGSS}, the natural   domains of definition  of quaternionic slice regular functions are the  so-called  symmetric slice domains, which play
for quaternionic slice regular functions the role played by domains of holomorphy for holomorphic functions of several complex variables. This is also the case for octonionic slice regular functions.

\begin{definition} \label{de: domain}
Let $\Omega$ be a domain in $\mathbb O $.

1. $\Omega$ is called a \textit{slice domain}  if it intersects the real axis and if  for every $I \in \mathbb S $, $\Omega_I$  is a domain in $ \mathbb C_I $.

2. $\Omega$ is called an \textit{axially symmetric domain} if for every point  $x + yI \in \Omega$, with  $x,y \in \mathbb R $ and $I\in \mathbb S$, the entire 6-dimensional sphere $x + y\mathbb S$ is contained in $\Omega $.
\end{definition}

A domain in $\mathbb O$ is called a \textit{symmetric slice domain} if it is not only a slice domain, but also an axially symmetric domain. By the very definition, an open ball $B(0,R)$ is  a typical  symmetric slice domain.  For slice regular functions a natural definition of slice derivative is given as follows:
\begin{definition} \label{de: derivative}
Let $f :\Omega \rightarrow \mathbb O$  be a slice regular function. For each $I\in\mathbb S$, the $I$- derivative of $f$ at $w=x+yI$
is defined by
$$\partial_I f(x+yI):=\frac{1}{2}\left(\frac{\partial}{\partial x}-I\frac{\partial}{\partial y}\right)f_I (x+yI)$$
on $\Omega_I$. The \textit{slice derivative} of $f$ is the function $f'$ defined by $\partial_I f$ on $\Omega_I$ for all $I\in\mathbb S$.
 \end{definition}

From the very definition and  Artin's theorem for alternative algebras (Theorem \ref{Artin-thm}) mentioned as before, the slice derivative of a slice regular function $f:\Omega \rightarrow \mathbb O$ is still slice regular so that we can iterate the differentiation to obtain the $n$-th
slice derivative
$$\partial^{n}_I f(w)=\frac{\partial^{n} f}{\partial x^{n}}(w),\quad\,\forall \,\, n\in \mathbb N,$$
where $w=x+yI\in \Omega$. In what follows, for the sake of simplicity, we will denote the $n$-th slice derivative by $f^{(n)}$ for every $n\in \mathbb N$.
Now it is easy to see that the coefficient $a_n$ appeared in (\ref{Taylor expansion on ball}) is exactly $f^{(n)}(0)/n!$. We will also omit the term `slice' when referring to slice regular functions.

Since  slice regularity does not keep under pointwise product of two regular functions, a new multiplication operation, called the regular product (or $\ast$-product), will be defined by means of a suitable modification of the usual one subject to the non-commutativity and the non-associativity of $\mathbb O$, based on the following splitting lemma (compare \cite[Lemma 2.7]{GS50}; also \cite[Lemma 2.4]{Ghiloni3}), which is more convenient for our subsequent arguments.

\begin{lemma}\label{eq:Splitting}
Let $f$ be a regular function on a  domain $\Omega\subseteq \mathbb O$. Then for any $I, J, K\in\mathbb S$ with  $I, J, IJ, K$ being mutually perpendicular with respect to the standard Euclidean inner product on $\mathbb O$, there exist four holomorphic functions $F_k: \Omega_I\rightarrow\mathbb C_I$, $k=1,2,3,4$ such that
\begin{equation}\label{splitting relation}
f_I(z)=F_1(z)+F_2(z)J+\big(F_3(z)+\overline{F_4(z)}J\big)K
\end{equation}
for all $z\in\Omega_I$.
\end{lemma}

\begin{proof}
The well-known Cayley-Dickson process guarantees the existence of four functions $F_k: \Omega_I\rightarrow\mathbb C_I$, $k=1,2,3,4$ such that  equality (\ref{splitting relation}) holds for all $z\in\Omega_I$. Now it remains to verify the holomorphy of these four functions $F_k, k=1,2,3,4$. By (\ref{Generation rule02}), for every $z\in\Omega_I$,
\begin{equation}\label{holomorphy-verification}
\begin{split}
\bar{\partial}_I f(z)
&=\bar{\partial}_IF_1(z)+\bar{\partial}_IF_2(z)J
+\Big(\big(F_3(z)+\overline{F_4(z)}J \big)\bar{\partial}_{I }\Big)K
\\
&=\bar{\partial}_IF_1(z)+\bar{\partial}_IF_2(z)J
+\big(\bar{\partial}_IF_3(z)+\partial_I\overline{F_4(z)}J\big)K.
\end{split}
\end{equation}
Thus $\bar{\partial}_I f(z)=0$ implies $\bar{\partial}_IF_k(z)=0$, proving the holomorphy of these four functions $F_k, k=1,2,3,4$. The proof is complete.
\end{proof}

Let $\Omega\subseteq \mathbb O$ be a symmetric slice domain and $I, J,  K\in\mathbb S$ be such that  $I, J, IJ, K$ are mutually perpendicular with respect to the standard Euclidean inner product on $\mathbb O$. Let $f$ and $g$ be two regular functions on   $\Omega\subseteq \mathbb O$. Then the splitting lemma  above guarantees the existence of eight holomorphic functions $F_k, G_k: \Omega_I\rightarrow \mathbb C_I$, $k=1,2,3,4$ such that for all $z\in\Omega_I$,
\begin{equation*}\label{splitting of f}
f_I(z)=F_1(z)+F_2(z)J+\big(F_3(z)+\overline{F_4(z)}J\big)K
\end{equation*}
and
\begin{equation*}\label{splitting of g}
g_I(z)=G_1(z)+G_2(z)J+\big(G_3(z)+\overline{G_4(z)}J\big)K
\end{equation*}
Following the approach in \cite{CGSS}, we define the function $f_I\ast g_I:\Omega_I\rightarrow\mathbb O$ as
\begin{equation}\label{def-regular product}
f_I\ast g_I(z):=H_1(z)+H_2(z)J+\big(H_3(z)+\overline{H_4(z)}J\big)K,
\end{equation}
where
$$H_1(z)=F_1(z)G_1(z)-F_2(z)\overline{G_2(\overline{z})}
-F_3(z)\overline{G_3(\overline{z})}-F_4(z)\overline{G_4(\bar{z})}, $$
$$H_2(z)=F_1(z)G_2(z)+F_2(z)\overline{G_1(\bar{z})}
+\overline{F_3(\overline{z})}\,\overline{G_4(\overline{z})}
-\overline{F_4(\overline{z})}\,\overline{G_3(\overline{z})}, $$
$$H_3(z)=F_1(z)G_3(z)-\overline{F_2(\overline{z})}\,\overline{G_4(\overline{z})}
+F_3(z)\overline{G_1(\overline{z})}+
\overline{F_4(\overline{z})}\,\overline{G_2(\overline{z})}, $$
$$H_4(z)=F_1(z)G_4(z)+\overline{F_2(\bar{z})}\,\overline{G_3(\overline{z})}
-\overline{F_3(\bar{z})}\,\overline{G_2(\overline{z})}+F_4(z)\overline{G_1(\overline{z})}. $$
Then $f_I\ast g_I(z)$ is obviously holomorphic satisfying $f_I\ast g_I(z)=f(z)g(z)$ (independent of the choice of $I\in\mathbb S$) for all $z\in\Omega\cap\mathbb R$. Therefore, $f_I\ast g_I$ admits a \textit{unique} regular extension to $\Omega$, independent of the choice of $I\in\mathbb S$,  via the formula (2) in \cite[Proposition 6]{Ghiloni1}, analogous to  \cite[Lemma 4.4]{CGSS}. We denote  by ${\rm{ext}}(f_I\ast g_I)$ this unique regular extension of $f_I\ast g_I$.

\begin{definition}\label{R-product}
Let $f$ and $g$ be two regular functions on  a symmetric slice domain $\Omega\subseteq \mathbb O$. Then the regular function
$$f\ast g(w):={\rm{ext}}(f_I\ast g_I)(w)$$
defined as the extension of (\ref{def-regular product}) is called the \textit{regular product} (or \textit{$\ast$-product}) of $f$ and $g$.
\end{definition}

\begin{remark}\label{remark on R-product01}
In the special case that $\Omega=B(0,R)$, there is a more direct way of defining the regular product.
Let $f$, $g:B(0,R)\rightarrow \mathbb O$ be two regular functions and let
\begin{equation*}\label{def of regular product}
f(w)=\sum\limits_{n=0}^{\infty}w^na_n,\qquad g(w)=\sum\limits_{n=0}^{\infty}w^nb_n
\end{equation*}
be their power series expansions. The regular product  of $f$ and $g$ given in Definition \ref{R-product} is coherent with  the one given by
$$f\ast g(w):=\sum\limits_{n=0}^{\infty}w^n\bigg(\sum\limits_{k=0}^n a_kb_{n-k}\bigg).$$
This follows from the identity principle and the fact that these two products defined in  these two ways are exactly the usual pointwise product  when they are restricted to $B(0,R)\cap\mathbb R$, as one can patiently verify.
\end{remark}

\begin{remark}\label{remark on R-product02}
When $\Omega\subseteq \mathbb O$ is a symmetric slice domain, the same reason as in the preceding remark also shows that the product defined in \cite[Definition 9]{Ghiloni1} coincides with the one in Definition \ref{R-product} provided all the considered functions are  regular.
\end{remark}

\begin{remark}\label{remark on R-product03}
Notice that the regular product is obviously distributive, but in general non-commutative and non-associative, since the underlying algebra $\mathbb O$ is  non-commutative and non-associative. However, it is commutative and associative in some special cases. For instance, let $f$ and $g$ be two regular functions on  a symmetric slice domain $\Omega\subseteq \mathbb O$ and satisfy the additional condition that $f(\Omega_I)\subseteq\mathbb C_I$ and $g(\Omega_I)\subseteq\mathbb C_I$ for some $I\in\mathbb S$, then from Artin's theorem for alternative algebras (Theorem \ref{Artin-thm}), Remark \ref{remark on R-product01} and the identity principle it follows that
$$f\ast g=g\ast f$$
and
$$(f\ast g)\ast h=f\ast(g\ast h)$$
for every regular function $h$ on $\Omega$. Moreover, if $f$ is further slice preserving (i.e. $f(\Omega_I)\subseteq \mathbb C_I$ for every $I\in\mathbb S$), then
$$f\ast h=fh=h\ast f.$$
\end{remark}

Again let $\Omega\subseteq \mathbb O$ be a symmetric slice domain and $I, J, K\in\mathbb S$ be such that  $I, J, IJ, K$ are mutually perpendicular with respect to the standard Euclidean inner product on $\mathbb O$. Let $f$ be a regular function on   $\Omega\subseteq \mathbb O$. Then the splitting lemma (Lemma \ref{eq:Splitting}) guarantees the existence of four holomorphic functions $F_k: \Omega_I\rightarrow \mathbb C_I$, $k=1,2,3,4$ such that for all $z\in\Omega_I$,
\begin{equation*}\label{splitting of f}
f_I(z)=F_1(z)+F_2(z)J+\big(F_3(z)+\overline{F_4(z)}J\big)K.
\end{equation*}
We define two functions $f^c_I, f^s_I:\Omega_I\rightarrow \mathbb O$ as
\begin{equation}\label{con-def}
 f^c_I(z):=\overline{F_1(\bar{z})}-F_2(z)J-\big(F_3(z)+\overline{F_4(z)}J\big)K,
\end{equation}
and
\begin{equation}\label{sym-def}
f^s_I(z):= f_I\ast f^c_I(z)
=\sum_{k=1}^{4}F_k(z)\overline{F_k(\bar{z})}= f^c_I\ast f_I(z).
\end{equation}
Here the second equality in (\ref{sym-def}) follows from (\ref{def-regular product}).
Then both $f^c_I(z)$ and $f^s_I(z)$ are obviously holomorphic satisfying $f^c_I(z)=\overline{f(z)}$ and $f^s_I(z)=|f(z)|^2$ (independent of the choice of $I\in\mathbb S$) for all $z\in\Omega\cap\mathbb R$. Therefore, they  admit respectively a \textit{unique} regular extension to $\Omega$, independent of the choice of $I\in\mathbb S$. We denote them  by ${\rm{ext}}(f^c_I)$ and ${\rm{ext}}(f^s_I)$, respectively.

\begin{definition}\label{con-sym-def}
Let $f$  be a regular function on  a symmetric slice domain $\Omega\subseteq \mathbb O$. Then the regular function
$$f^c(w):={\rm{ext}}(f^c_I)(w)$$
defined as the extension of (\ref{con-def}) is called the \textit{regular conjugate} of $f$, and the regular function
$$f^s(w):={\rm{ext}}(f^s_I)(w)=f\ast f^c(w)=f^c\ast f(w)$$
is called the \textit{symmetrization} of $f$.
\end{definition}

\begin{remark}\label{remark on con-sym01}
As before, one can also prove that for octonionic regular functions on symmetric slice domains in $\Omega\subseteq \mathbb O$, the notions given in Definition \ref{con-sym-def} for regular conjugate and symmetrization coincide essentially with those introduced in \cite[Definition 11]{Ghiloni1}. In the special case that $\Omega=B(0,R)$, there is also an equivalent way of defining the regular conjugate and the symmetrization of regular functions, which goes as follows.
Let $f:B(0,R)\rightarrow \mathbb O$ be a regular function with the  power series expansion
\begin{equation*}\label{equi-def}
f(w)=\sum\limits_{n=0}^{\infty}w^na_n.
\end{equation*}
Then the regular conjugate and the symmetrization of $f$ are respectively given by
$$f^c(w)=\sum\limits_{n=0}^{\infty}w^n\overline{a}_n,$$
and
$$f^s(w)=f\ast f^c(w)=f^c\ast f(w)=\sum\limits_{n=0}^{\infty}w^n\bigg(\sum\limits_{k=0}^n a_k\overline{a}_{n-k}\bigg).$$
One can easily verify that these two definitions are the same as those in Definition \ref{con-sym-def}.
\end{remark}

\begin{remark}\label{remark on con-sym02}
From (\ref{sym-def}) one immediately deduces that the  symmetrization $f^s$ of every regular function $f$ on a symmetric slice domain  $\Omega\subseteq \mathbb O$ is slice preserving, i.e. $f^s(\Omega_I)\subseteq \mathbb C_I$ for every $I\in\mathbb S$.
\end{remark}

Both the regular conjugate and  the symmetrization are well-behaved with respect to the  regular product.

\begin{proposition}\label{beh-con-sym}
Let $f$ and $g$ be two regular functions on  a symmetric slice domain $\Omega\subseteq \mathbb O$. Then $(f^c)^c=f$, $(f\ast g)^c=g^c\ast f^c$ and
\begin{equation}\label{beh-symmetrization}
(f\ast g)^s=f^sg^s=g^sf^s.
\end{equation}
\end{proposition}

\begin{proof}
We only prove (\ref{beh-symmetrization}), since the remaining is obvious in virtue of (\ref{def-regular product}) and (\ref{con-def}). The power series case of (\ref{beh-symmetrization}) was proved in \cite[Lemma 2]{Ghiloni2}, and the general case follows immediately from the former case and the identity principle.
\end{proof}

Now we can use the notions of regular conjugate and symmetrization introduced above to define the regular reciprocal of a regular function:

\begin{definition}\label{def-regular reciprocal}
Let $f$  be a non-identically vanishing  regular function on  a symmetric slice domain $\Omega\subseteq \mathbb O$ and $\mathcal{Z}_{f^s}$   the set of zeros of its symmetrization $f^s$. We define the \textit{regular reciprocal} of $f$ as the regular function $f^{-\ast}:\Omega\setminus \mathcal{Z}_{f^s}\rightarrow \mathbb O$ given by
\begin{equation}\label{exp-regular reciprocal}
f^{-\ast}(w):=f^s(w)^{-1}f^c(w).
\end{equation}
\end{definition}

The function $f^{-\ast}$ defined in (\ref{exp-regular reciprocal}) deserves the name of regular reciprocal of $f$ due to the following:

\begin{proposition}
Let $f$  be a  non-identically vanishing regular function on  a symmetric slice domain $\Omega\subseteq \mathbb O$ and $\mathcal{Z}_{f^s}$   the set of zeros of its symmetrization $f^s$. Then
$$f^{-\ast}\ast f=f\ast f^{-\ast}=1$$
on $\Omega\setminus \mathcal{Z}_{f^s}$.
\end{proposition}

We conclude this section with the following simple proposition.
\begin{proposition}
Let $f$ and $g$  be two non-identically vanishing regular function on  a symmetric slice domain $\Omega\subseteq \mathbb O$. Then
$$(f\ast g)^{-\ast}=g^{-\ast}\ast f^{-\ast}$$
on $\Omega\setminus (\mathcal{Z}_{f^s}\cup \mathcal{Z}_{g^s})$.
\end{proposition}

\begin{proof}
The result follows from Remarks \ref{remark on R-product03} and \ref{remark on con-sym02}, together with Proposition \ref{beh-con-sym}:
$$(f\ast g)^{-\ast}=(f^sg^s)^{-1}(g^c\ast f^c)
=\big((g^s)^{-1}g^c\big)\ast\big((f^s)^{-1}f^c\big)=g^{-\ast}\ast f^{-\ast}.$$
\end{proof}

\section{Proof of Theorem \ref{BSL}} \label{Proof of Theorem BSL}
In this section, we shall give a proof of Theorem \ref{BSL}. Before presenting the details, we need some auxiliary results.

Let $\Omega\subseteq\mathbb O$ be a symmetric slice domain. For each regular function $f:\Omega\rightarrow \mathbb O$ and each $\xi\in\Omega$, an argument similar to the one in the proof of \cite[Proposition 3.17]{GSS}, together with the splitting lemma (Lemma \ref{eq:Splitting}) and Artin's theorem for alternative algebras (Theorem \ref{Artin-thm}), guarantees  the existence of a unique regular function on $\Omega$, denoted by $R_{\xi}f$, such that
\begin{equation}\label{de:Rf1}
f(w)-f(\xi)=(w-\xi)\ast R_{\xi}f(w), \qquad \forall \, w\in\Omega.
\end{equation}
Applying the same procedure to $R_{\xi}f$ at the point $\overline{\xi}$ yields
\begin{equation}\label{de:Rf2}
\begin{split}
f(w)=&f(\xi)+(w-\xi)\ast \Big(R_{\xi}f(\bar{\xi})+(w-\bar{\xi}\,)\ast R_{\overline{\xi}}R_{\xi}f(w)\Big) \\
=&f(\xi)+(w-\xi)R_{\xi}f(\bar{\xi})+\Delta_{\xi}(w)R_{\overline{\xi}}R_{\xi}f(w), \qquad \forall \, w\in\Omega,
\end{split}
\end{equation}
where $\Delta_{\xi}(w):=(w-\xi)\ast(w-\bar{\xi}\,)=w^2-2w{\rm{Re}}(\xi)+|\xi|^2$ is called the \textit{characteristic polynomial} of $\xi$ or the \textit{symmetrization} of $w-\xi$, and the second equality in (\ref{de:Rf2}) follows from Remark \ref{remark on R-product03}.

From the very definition and (\ref{de:Rf1}), one can see that $R_{\xi}f(\bar{\xi})$ is exactly the \textit{sphere derivative} $\partial_sf(\xi)$  of $f$ at the point $\xi$:
$$\partial_sf(\xi):=\big(2{\rm{Im}}(\xi)\big)^{-1}\big(f(\xi)-f(\overline{\xi})\big)=R_{\xi}f(\bar{\xi}).$$
In addition, for every $v\in\partial \mathbb B$ and every $t\in\mathbb R$ small enough, replacing $w$ by $\xi+tv$ in (\ref{de:Rf2}) yields
$$f(\xi+tv)-f(\xi)=tv\partial_sf(\xi)+t\big(tv^2+(\xi v-v\bar{\xi}\,)\big)R_{\overline{\xi}}R_{\xi}f(\xi+tv),$$
from which the following lemma immediately follows.
\begin{lemma}\label{D-derivative on O}
Let $f$ be a regular function on a  symmetric slice domain $\Omega\subseteq \mathbb O$ and $\xi\in\Omega$. Then for every $v\in\partial \mathbb B$, the directional derivative of $f$ along $v$ at $\xi$ is given by
\begin{equation}\label{eq:D-derivative}
\frac{\partial f}{\partial v}(\xi):=\lim\limits_{\mathbb R\ni t\rightarrow 0}\frac{f(\xi+tv)-f(\xi)}{t}=v\partial_sf(\xi)+(\xi v-v\overline{\xi}\,)R_{\overline{\xi}}R_{\xi}f(\xi).
\end{equation}
In particular, it holds that
\begin{equation}\label{derivative-relation}
f'(\xi)=\partial_sf(\xi)+2{\rm{Im}}(\xi)R_{\overline{\xi}}R_{\xi}f(\xi).
\end{equation}
\end{lemma}

Also, the following lemma is needed in the proof of Theorem \ref{BSL}.

\begin{lemma}\label{lem: moudulus inequality}
Let $f$ be a   regular self-mapping of the open unit ball $\mathbb B$. Then the inequality
\begin{equation}\label{moudulus inequality}
\frac{1-|f(w)|^2}{1-|w|^2}\geq\frac{\big|1-\big\langle f(0), f(w)\big\rangle\big|^2}{1-|f(0)|^2}
\end{equation}
holds for all $w\in\mathbb B$.
\end{lemma}

\begin{proof}
Fix an arbitrary  point  $w\in\mathbb B$, let $I\in\mathbb S$ be such that $w\in \mathbb B_I$. Then by the splitting lemma (Lemma \ref{eq:Splitting}), we can find $J$ and $K$ in $\mathbb S$, such that $I, J, IJ, K$ are mutually perpendicular with respect to the standard Euclidean inner product on $\mathbb O$ and there are  four holomorphic functions
$F_k: \mathbb B_I\rightarrow\mathbb B_I$, $k=1,2,3,4$,  such that
\begin{equation}\label{splitting}
f_I(z)=F_1(z)+F_2(z)J+\big(F_3(z)+\overline{F_4(z)}J\big)K, \qquad \forall\, z\in\mathbb B_I.
\end{equation}
Let  $B^4\subset\mathbb C_I^4$ be the open unit ball. We consider the holomorphic mapping $F:\mathbb B_I\rightarrow \mathbb C_I^4$ given by
$$F(z):=\big(F_1(z), F_2(z), F_3(z), F_4(z)\big), $$
which maps $\mathbb B_I$ into $B^4$
in virtue of the fact that
$$|F(z)|^2=\sum\limits_{k=1}^4|F_k(z)|^2=|f(z)|^2<1$$
for all $z\in\mathbb B_I$. Now it follows from the classical Schwarz-Pick lemma (see e.g. \cite[Theorem 8.1.4]{Rudin}) that
\begin{equation}\label{norm-Schwarz}
 \frac{\big|1-\big\langle F(0), F(z)\big\rangle_{\mathbb C_{I}^4}\big|^2}{\big(1-|F(0)|^2\big)\big(1-|F(z)|^2\big)}\leq\frac{1}{1-|z|^2},
 \qquad \forall\, z\in\mathbb B_{I},
\end{equation}
where $\langle \,\, ,\,\rangle_{\mathbb C_{I}^4}$ denotes the standard Hermitian inner product on $\mathbb C_{I}^4$, i.e. for any two vectors $\alpha=(\alpha_1,\ldots, \alpha_4)$, $\beta=(\beta_1,\ldots, \beta_4)\in \mathbb C_{I}^4$,
$$\langle \alpha, \beta\rangle_{\mathbb C_{I}^4}=\sum_{k=1}^4\alpha_k \overline{\beta}_k.$$
Once notice that $\big\langle f(0), f(w)\big\rangle={\rm{Re}}\big(\langle F(0), F(z)\rangle_{\mathbb C_{I}^4}\big)\in\mathbb R$, inequality  (\ref{moudulus inequality}) immediately follows from (\ref{norm-Schwarz}).
\end{proof}


Now we come to prove Theorem $\ref{BSL}$.

\begin{proof}[Proof of Theorem $\ref{BSL}$]
We first prove the assertion (i). The proof of the first equality in (\ref{Schwarz ineq1})
 is essentially the same as the corresponding part in  the proof of \cite[Theorem 4]{WR}.
First, it follows from inequality  (\ref{moudulus inequality}) that the directional derivative of $|f|^2$ along $\xi$ at the boundary point $\xi\in\partial\mathbb B$ satisfies that
\begin{equation}\label{radial der}
\frac{\partial |f|^2}{\partial \xi}(\xi)\geq 2\frac{\big|1-\big\langle f(0), f(\xi)\big\rangle\big|^2}{1-|f(0)|^2}.
\end{equation}
However,
\begin{equation}\label{tangent der}
\frac{\partial |f|^2}{\partial \tau}(\xi)=0, \qquad \forall\, \,\tau\in T_{\xi}(\partial \mathbb B)\cong \mathbb R^7.
\end{equation}
Indeed, for each unit tangent vector $\tau\in T_{\xi}(\partial \mathbb B)$, take a smooth curve
$\gamma:(-1,1)\rightarrow\overline{\mathbb B}$ such that
$$\gamma(0)=\xi,\quad \gamma'(0)=\tau.$$
By definition we have
$$\frac{\partial |f|^2}{\partial \tau}(\xi)=\left.\bigg(\frac{d}{dt}\big|f(\gamma(t))\big|^2\bigg)\right|_{t=0}=0,$$
since the function $|f(\gamma(t))\big|^2$ in $t$ attains its maximum at the point $t=0$.

In view of Lemma \ref{D-derivative on O}, we have
$$\frac{\partial f}{\partial v}(\xi)=v\partial_sf(\xi)+(\xi v-v\overline{\xi}\,)R_{\overline{\xi}}R_{\xi}f(\xi),\qquad \forall\,v\in\partial \mathbb B,$$
from which and Lemma \ref{unitary multipliers} it follows that
\begin{equation}\label{der-relation1}
\begin{split}
\frac{\partial |f|^2}{\partial v}(\xi)
&=2\Big\langle\frac{\partial f}{\partial v}(\xi),f(\xi)\Big\rangle\\
&=2\Big\langle v\partial_sf(\xi)+(\xi v-v\overline{\xi}\,)R_{\bar{\xi}}R_{\xi}f(\xi),f(\xi)\Big\rangle\\
&=2\Big\langle v,f(\xi)\overline{\partial_sf(\xi)}+
\bar{\xi}\Big(f(\xi)\overline{R_{\bar{\xi}}R_{\xi}f(\xi)}\Big)-
\Big(f(\xi)\overline{R_{\bar{\xi}}R_{\xi}f(\xi)}\Big)\xi\Big\rangle\\
&=:2\big(A+B\big),
\end{split}
\end{equation}
where
\begin{equation}\label{exp-A}
\begin{split}
A&=\Big\langle v,\,f(\xi)\Big(\overline{\partial_sf(\xi)
 -\bar{\xi}R_{\bar{\xi}}R_{\xi}f(\xi)}\Big)\Big\rangle \\
&=\Big\langle v,\,f(\xi)\Big(\overline{f'(\xi)-\xi R_{\bar{\xi}}R_{\xi}f(\xi)}\Big)\Big\rangle,
\end{split}
\end{equation}
and
\begin{equation}\label{exp-B}
B=\Big\langle v,\,\bar{\xi}\Big(f(\xi)\overline{R_{\bar{\xi}}R_{\xi}f(\xi)}\Big)-\big[f(\xi), \overline{R_{\bar{\xi}}R_{\xi}f(\xi)}, \xi\big]\Big\rangle.
\end{equation}
The second equality in (\ref{exp-A}) follows from  equality (\ref{derivative-relation}).
Substituting the following simple equalities
\begin{equation*}
\begin{split}
f(\xi)\Big(\overline{R_{\bar{\xi}}R_{\xi}f(\xi)}\bar{\xi}\Big)
&=\Big(f(\xi)\overline{R_{\bar{\xi}}R_{\xi}f(\xi)}\Big)
\bar{\xi}-\big[f(\xi), \overline{R_{\bar{\xi}}R_{\xi}f(\xi)},\bar{\xi}\,\big]\\
&=\Big(f(\xi)\overline{R_{\bar{\xi}}R_{\xi}f(\xi)}\Big)
\bar{\xi}+\big[f(\xi), \overline{R_{\bar{\xi}}R_{\xi}f(\xi)},\xi\,\big]
\end{split}
\end{equation*}
into  the second equality in (\ref{exp-A}) yields
\begin{equation}\label{exp-A01}
\begin{split}
A&=\Big\langle v,\,f(\xi)\overline{f'(\xi)}-\Big(f(\xi)\overline{ R_{\bar{\xi}}R_{\xi}f(\xi)}\Big)\bar{\xi}-\big[f(\xi), \overline{R_{\bar{\xi}}R_{\xi}f(\xi)},\xi\,\big]\Big\rangle.
\end{split}
\end{equation}
Substituting (\ref{exp-B}) and (\ref{exp-A01}) into (\ref{der-relation1}) gives
\begin{equation}\label{der-relation2}
\begin{split}
\frac{\partial |f|^2}{\partial v}(\xi)
&=2\Big\langle v,f(\xi)\overline{f'(\xi)}+\big[\bar{\xi}, f(\xi)\overline{R_{\bar{\xi}}R_{\xi}f(\xi)}\,\big]-2\big[f(\xi), \overline{R_{\bar{\xi}}R_{\xi}f(\xi)}, \xi\big]\Big\rangle\\
&=2\Big\langle v,f(\xi)\overline{f'(\xi)}+\big[\bar{\xi}, f(\xi)\overline{R_{\bar{\xi}}R_{\xi}f(\xi)}\,\big]+2\big[\xi, f(\xi), R_{\bar{\xi}}R_{\xi}f(\xi)\big]
\Big\rangle.
\end{split}
\end{equation}
Now it follows from $(\ref{tangent der})$ and $(\ref{der-relation2})$ that
$$f(\xi)\overline{f'(\xi)}+\big[\bar{\xi}, f(\xi)\overline{R_{\bar{\xi}}R_{\xi}f(\xi)}\,\big]+2\big[\xi, f(\xi), R_{\bar{\xi}}R_{\xi}f(\xi)\big]\perp
 T_{\xi}(\partial \mathbb B)$$
so that in view of $(\ref{radial der})$ and $(\ref{der-relation2})$,
\begin{equation*}
\begin{split}
\frac{\partial |f|}{\partial \xi}(\xi)
&=\overline{\xi}\Big(f(\xi)\overline{f'(\xi)}+\big[\bar{\xi}, f(\xi)\overline{R_{\bar{\xi}}R_{\xi}f(\xi)}\,\big]+2\big[\xi, f(\xi), R_{\bar{\xi}}R_{\xi}f(\xi)\big]\Big)
\\
&\geq \frac{\big|1-\big\langle f(0), f(\xi)\big\rangle\big|^2}{1-|f(0)|^2},
\end{split}
\end{equation*}
which completes the proof of $(\ref{Schwarz ineq1})$.

To prove $(\ref{Schwarz ineq2})$, notice first that the first equality in $(\ref{Schwarz ineq2})$ directly follows from  $(\ref{Schwarz ineq1})$. It remains to prove the following inequality
$$\frac{\partial |f|}{\partial \xi}(\xi)\geq\frac{2}{1+{\rm{Re}}f'(0)}.$$
To this end, let $I\in\mathbb S$ be such that $\xi\in \partial\mathbb B\cap \mathbb C_I$. Then by the splitting lemma (Lemma \ref{eq:Splitting}), we can find $J$ and $K$ in $\mathbb S$, such that $I, J, IJ, K$ are mutually perpendicular and if $\mathcal{H}$ is the subspace of $\mathbb O$ generated by $\{1, I, J, IJ\}$, then there are two regular functions
$F: \mathbb B\cap\mathcal{H}\rightarrow\mathbb B\cap\mathcal{H}$ and $G: \mathbb B\cap\mathcal{H}\rightarrow\mathbb B\cap\mathcal{H}K$  such that
$$f(w)=F(w)+G(w),\qquad \forall\, w\in\mathbb B\cap\mathcal{H}.$$
 Then for each $w\in\mathbb B\cap\mathcal{H}$, we have
\begin{equation}\label{squared-norm}
|f(w)|^2=|F(w)|^2+|G(w)|^2
\end{equation}
and
$$f'(w)=F'(w)+G'(w).$$
Moreover,
$$F(\xi)=\xi, \qquad G(\xi)=0, \qquad {\rm{Re}}\,f'(0)={\rm{Re}}\,F'(0).$$
Now it follows from Corollary \ref{Cor-Schwarz} below that
$$\frac{\partial |f|}{\partial \xi}(\xi)
=\frac{\partial |F|}{\partial \xi}(\xi)
=F'(\xi)-\big[\xi,R_{\bar{\xi}}R_{\xi}F(\xi)\big]
\geq\frac{2}{1+{\rm{Re}}F'(0)}=\frac{2}{1+{\rm{Re}}f'(0)}.$$
If equality holds for inequality in (\ref{Schwarz ineq2}), then it again follows from Corollary \ref{Cor-Schwarz} below that
\begin{equation}\label{F-expression}
F(w)=w\big(1-wa\bar{\xi}\,\big)^{-\ast}\ast\big(w-a\xi\big)\bar{\xi}
\qquad \forall\, w\in\mathbb B\cap\mathcal{H},
\end{equation}
for some constant $a\in [-1,1)$. Furthermore, it follows from equality in (\ref{squared-norm}) that
$$|G(w)|^2=|f(w)|^2-|F(w)|^2\leq 1-|F(w)|^2, \qquad \forall\, w\in\mathbb B\cap\mathcal{H},$$
which together with (\ref{F-expression}) implies that $G\equiv0$, in virtue of the maximum principle (Theorem \ref{MP} below),
and hence
$$f(w)={\rm{ext}}\,F(w)=w\big(1-wa\bar{\xi}\,\big)^{-\ast}\ast\big(w\bar{\xi}-a\big),\qquad \forall\, w\in\mathbb B.$$
Therefore, the equality in inequality (\ref{Schwarz ineq2}) can hold only for regular self-mappings  of  the form (\ref{f-expression}), and a direct calculation shows that it does indeed hold for all such regular self-mappings. Now  the proof is complete.
\end{proof}

\begin{remark}\label{Non-real}
It is worth remarking here that, as in the quaternionic setting, the Lie bracket
in (\ref{Schwarz ineq2}) does not necessarily vanish
and $f'(\xi)$ may not be a real number. The following example comes an explicit counterexample.
\end{remark}

\begin{example}\label{Non-real-example}
Fix two imaginary units $I$, $J\in \mathbb S$ with $I\bot J$. Set
$$\varphi(w)=w\big(1+wI/2\big)^{-\ast}\ast\big(I/2-w\big).$$
Then the restriction $\varphi_I$ of $\varphi$ to $\mathbb B_I$ is a holomorphic Blaschke product of order 2 so that $\varphi$ is a regular self-mapping of $\mathbb B$, in virtue of Proposition \ref{convex combination identity}. Define another  regular function $f$ on $\mathbb B$ given by
$$f(w)=\varphi(w)\ast J=w\big(w^2+4\big)^{-1}\Big(2(w^2+1)(IJ)-3wJ\Big).$$
We claim that $f$ maps $\mathbb B$ into $\mathbb B$. We argue by contradiction and suppose that there is a point $\omega_0\in\mathbb O\setminus \mathbb B$ such that $f-\omega_0=(\varphi+\omega_0 J)\ast J$ has a zero in $\mathbb B$. By \cite[Corollary 25]{Ghiloni1},
$$\bigcup_{w\in \mathcal{Z}_{f-\omega_0}}\mathbb S_{w}=
\bigcup_{w\in \mathcal{Z}_{\varphi+\omega_0 J}}\mathbb S_{w}.$$
This shows that $\varphi+\omega_0J$ also has a zero in $\mathbb B$, contradicting the fact that $\varphi(\mathbb B)\subseteq \mathbb B$. Therefore, $f(\mathbb B)\subseteq \mathbb B$.
Moreover, it is evident that $f$ is regular on $\overline{\mathbb B}$ satisfying both $f(0)=0$ and $f(J)=J$. Thus $f$ verifies all the assumptions in Theorem \ref{BSL} (ii). However, we find that $f'(J)$ is indeed not a real number. In fact, a straightforward calculation shows that
$$f'(J)=\frac43(2-IJ)\notin \mathbb R, \qquad R_{-J}R_Jf(J)=\frac23(I-2J), $$
while
$$f'(J)-\big[J, \, R_{-J}R_Jf(J)\big]=\frac83>1$$
as predicated by Theorem \ref{BSL} (ii). One can also shows that
$$\frac{\partial |f|}{\partial J}(J)=\frac83$$
using the obvious fact that $f_J(w)=\varphi_J(w)J$ for all $w\in\mathbb B_J$, together with Proposition \ref{convex combination identity}.
\end{example}

The regular functions of the form (\ref{f-expression}) are indeed self-mappings of the open unit ball $\mathbb B\subset\mathbb O$, due to the following result:

\begin{proposition}\label{convex combination identity}
Let $f$ be a regular function on a symmetric slice domain  $\Omega\subseteq \mathbb O$ such that $f(\Omega_I)\subseteq \mathbb C_I $ for some $I\in \mathbb S $.  Then
the convex combination identity
\begin{equation}\label{convex combination identity01}
\big|f(x+yJ)\big|^2 =\frac{1+\langle I,J\rangle}{2}\big|f(x+yI)\big|^2+
\frac{1-\langle I,J\rangle}{2}\big|f(x-yI)\big|^2
\end{equation}
holds for every $ x+yJ \in \Omega$.
\end{proposition}

\begin{proof}
The idea is essentially the same as in \cite{RW2}. Let $I\in \mathbb S$ be   as described in  the proposition. First, it is easy to verify that for every $J\in \mathbb S$, the set $\big\{1, I, I\wedge J, I(I\wedge J)\big\}$ is an orthogonal set of $\mathbb O\simeq\mathbb R^8$. By the representation formula for regular functions (cf. \cite[Proposition 6]{Ghiloni1}),
\begin{equation}\label{repesentaion}
f(w)=\frac{1}{2}\big(f(z)+f(\bar z)\big)-\frac{1}{2}J\Big(I\big(f(z)-f(\bar z)\big)\Big)
\end{equation}
for every  $ w=x+yJ \in \Omega$ with $z=x+yI$ and $\bar z=x-yI$. By assumption,  $f(\Omega_I)\subseteq \mathbb C_I $. This together with   Artin's theorem for alternative algebras (Theorem \ref{Artin-thm}) allows us to rewrite equality (\ref{repesentaion}) as
\begin{equation}\label{repesentaion001}
f(w)=\frac{1}{2}\big(f(z)+f(\bar z)\big)-\frac{1}{2}(JI)\big(f(z)-f(\bar z)\big)
\end{equation}

Taking modulus on both sides of (\ref{repesentaion001}) and applying Lemma \ref{unitary multipliers} to obtain
\begin{equation}\label{norm-repesentaion}
\begin{split}
|f(w)|^2
=&\frac14\Big(\big|f(z)+f(\bar z)\big|^2+\big|f(z)-f(\bar z)\big|^2\Big)\\
&-\frac12\Big\langle f(z)+f(\bar z), \, (JI)\big(f(z)-f(\bar z)\big)\Big\rangle\\
=&\frac12\Big(|f(z)|^2+|f(\bar z)|^2\Big)
-\frac12\Big\langle \big(f(z)+f(\bar z)\big)\big(\,\overline{f(z)}-\overline{f(\bar z)}\,\big), \,J I\Big\rangle\\
=&: \frac12\Big(|f(z)|^2+|f(\bar z)|^2\Big)-\frac{1}{2}A,
\end{split}
\end{equation}
where
\begin{equation}\label{B-computation01}
A=\Big\langle \big(f(z)+f(\bar z)\big)\big(\,\overline{f(z)}-\overline{f(\bar z)}\,\big), \,J I\Big\rangle.
\end{equation}
Recalling equality in (\ref{relation-inner-wedge}), an orthogonality consideration gives
\begin{equation}\label{B-computation02}
\begin{split}
A&=-\langle I, J\rangle\Big\langle \big(f(z)+f(\bar z)\big)\big(\,\overline{f(z)}-\overline{f(\bar z)}\,\big), \,1\Big\rangle\\
&=-\langle I, J\rangle\Big\langle f(z)+f(\bar z), \,f(z)-f(\bar z)\Big\rangle\\
&=-\langle I, J\rangle \Big(|f(z)|^2-|f(\bar z)|^2\Big).
\end{split}
\end{equation}
Now the desired equality (\ref{convex combination identity01}) immediately follows by substituting  (\ref{B-computation02}) into (\ref{norm-repesentaion}). The proof is complete.
 \end{proof}

\begin{remark}
Together with Proposition \ref{convex combination identity}, the argument used in Example \ref{Non-real-example} also shows that the regular functions $f$ of the from
$$f(w)=\big(1-w\overline{u}\big)^{-\ast}\ast\big(w-u\big)\ast v$$ with $u\in \mathbb B$ and $v\in \partial \mathbb B$ are regular self-mappings of $\mathbb B$.
\end{remark}

\section{Proofs of Theorems \ref{regular diam-LT}, \ref{slice diam-LT} and \ref{Poukka}}
\label{applications of Thm BSL}
We begin with a notion of regular diameter, which is intimately related to a new regular composition (cf. \cite{RW}).

\begin{definition}\label{regular-composition}
Let $u\in \mathbb O$ and $f:\mathbb B\rightarrow \mathbb O$   a regular function with  Taylor expansion
$$f(w)=\sum\limits_{n=0}^{\infty}w^na_n.$$
We define the \textit{regular composition} of $f$ with the regular function $w\mapsto wu$  to be
$$f_u(w):=\sum\limits_{n=0}^{\infty}(wu)^{\ast n}\ast a_n=\sum\limits_{n=0}^{\infty}w^n(u^na_n).$$
\end{definition}
If $|u|=1$, the radius of convergence of the series expansion for $f_u$ is the same as that for $f$. Moreover, if $u$ and $w_0$ lie in the same plane $\mathbb C_I$, then $u$ and $w_0$ commute, and hence $f_u(w_0)=f(uw_0)$. In particular,  if $u\in \mathbb R$, then $f_u(w)=f(uw)$ for every $w\in\mathbb B$.

\begin{definition}\label{regular-diameter}
Let $f$ be a regular function on $\mathbb B$ with   Taylor expansion
$$f(w)=\sum\limits_{n=0}^{\infty}w^na_n.$$ For each $r\in(0, 1)$, the \textit{regular diameter} of the image of $r\mathbb B$ under $f$ is defined to be
\begin{equation}\label{regular-diameter01}
\widetilde{d}\big(f(r\mathbb B)\big):=\max_{u,v\in \overline{\mathbb B}}\max_{|w|\leq r}|f_u(w)-f_v(w)|.
\end{equation}
The \textit{regular diameter} of the image of $\mathbb B$ under $f$ is defined to be
\begin{equation}\label{regular-diameter02}
\widetilde{d}\big(f(\mathbb B)\big):=\lim_{r\rightarrow 1^-}\widetilde{d}\big(f(r\mathbb B)\big).
\end{equation}
\end{definition}

Clearly, $\widetilde{d}\big(f(r\mathbb B)\big)$ is an increasing function of $r\in(0,1)$; hence the limit in (\ref{regular-diameter02}) always exists. Therefore, $\widetilde{d}\big(f(\mathbb B)\big)$ is well-defined. Moreover, in view of the following maximum principle for  regular functions, $\widetilde{d}\big(f(r\mathbb B)\big)/2r$ is an increasing function of $r\in(0,1)$ as well (see (\ref{diameter-quo01}) below).

\begin{theorem}\label{MP}
Let $f:\Omega\rightarrow \mathbb O$ be a regular function on a symmetric slice domain $\Omega\subseteq\mathbb O$. If there exist a $I\in\mathbb S$ such that the restriction $|f_I|$  of $|f|$ to $\Omega_I$ attains a local maximum at some point $w_0\in \Omega_I$, then $f$ is constant.
\end{theorem}

\begin{proof}
 We can split $f_{I}$ as $$f(z)=F_1(z)+F_2(z)J+\big(F_3(z)+\overline{F_4(z)}J\big)K, \qquad \forall\, z\in\Omega_I,$$
where $F_k: \Omega_I\rightarrow \mathbb C_I$, $k=1,2,3,4$, are four holomorphic functions, and $I, J, K$ enjoy the same property as   in the proof of Lemma \ref{lem: moudulus inequality}. Then the holomorphic mapping $F:\Omega_{I}\rightarrow \mathbb C_{I}^4$ given by
$$F(z):=\big(F_1(z), F_2(z), F_3(z), F_4(z)\big)$$
 satisfies that
$$|F(z)|^2=\sum\limits_{k=1}^4|F_k(z)|^2=|f(z)|^2$$
for all $z\in\Omega_{I}$. By assumption, $|F|$ attains a local maximum at the point $w_0\in \Omega_{I}$. Thus from the maximum principle for holomorphic mappings (cf. \cite[Theorem 2.8.3]{Klimek}) it immediately follows that $F$ is constant on $\Omega_{I}$, and $f$ is constant there as well, and in turn on $\Omega$ by the identity principle.
\end{proof}

We also need the following results.

\begin{proposition}\label{regular-diameter-relation}
Let $f$ be a regular function on $\mathbb B$. Then
\begin{equation}\label{regular-diameter-relation01}
 {\rm{diam}}\, f(\mathbb B)\leq \widetilde{d}\big(f(\mathbb B)\big)\leq 2\,{\rm{diam}}\, f(\mathbb B).
\end{equation}
\end{proposition}

\begin{lemma}\label{Constant lemma}
Let  $g$ be a regular function on $\mathbb B$ such that for each $w\in\mathbb B\setminus\{0\}$,
$$\big\langle I_w, g(w)\big\rangle=0, $$
where $I_{w}={\rm{Im}}\,w/|{\rm{Im}}\,w|$ is the pure imaginary unit identified by $w$. Then $g$ is a real constant function.
\end{lemma}

The proofs of Proposition \ref{regular-diameter-relation}  and Lemma \ref{Constant lemma} are completely the same as those of \cite[Propositions 3.8 and 3.4]{Gen-Sar}, and so we omit them. Now we are ready to prove Theorem \ref{regular diam-LT}.

\begin{proof}[Proof of Theorem $\ref{regular diam-LT}$]
The proof is partly the same as that of \cite[Theorem 3.9]{Gen-Sar}, the main difference being that we use some extra technical treatments together with Theorem \ref{BSL}, instead of \cite[Proposition 3.2]{Gen-Sar}, which is not enough for our purpose because of the non-associativity of octonions.

We first prove inequality (\ref{regular-diam01}). To this end, we take $u, v\in\overline{\mathbb B}$ and consider the following auxiliary function
$$g_{u,v}(w)=\frac12 w^{-1}\big(f_u(w)-f_u(w)\big).$$
Then $g_{u,v}$ is regular on $\mathbb B$ with
\begin{equation}\label{g-value}
g_{u,v}(0)=\frac12 (u-v)f'(0).
\end{equation}
Applying the maximum principle (Theorem \ref{MP}) to the  regular function $g_{u,v}$
 yields that for each $r\in(0,1)$, we can write
$$\max_{|w|\leq r}|g_{u,v}(w)|=\max_{|w|\leq r}\frac{|f_u(w)-f_v(w)|}{2|w|}
=\frac1{2r}\max_{|w|\leq r} |f_u(w)-f_v(w)|,$$
which implies that
\begin{equation}\label{diameter-quo01}
\frac{\widetilde{d}\big(f(r\mathbb B)\big)}{2r}
=\frac1{2r}\max_{u,v\in \overline{\mathbb B}}\max_{|w|\leq r}|f_u(w)-f_v(w)|
=\max_{u,v\in \overline{\mathbb B}}\max_{|w|\leq r}|g_{u,v}(w)|.
\end{equation}
Therefore, $\widetilde{d}\big(f(r\mathbb B)\big)/2r$ is an increasing function of $r\in(0,1)$ and so always not more than
$$\lim_{r\rightarrow1^-}\frac{\widetilde{d}\big(f(r\mathbb B)\big)}{2r}
=\frac12 \widetilde{d}\big(f(\mathbb B)\big)=1.$$

This means that
\begin{equation}\label{regular-diam10}
\widetilde{d}\big(f(r\mathbb B)\big)\leq 2r
\end{equation}
for each $r\in(0, 1)$, proving inequality (\ref{regular-diam01}). To prove inequality (\ref{regular-diam02}), consider the odd part of $f$
$$f_{odd}(w)=\frac12\big(f(w)-f(-w)\big),$$
which is regular on $\mathbb B$ satisfying both $f_{odd}(0)=0$ and
$$|f_{odd}(w)|=\frac12|f(w)-f(-w)|\leq\frac12 \widetilde{d}\big(f(\mathbb B)\big)=1$$
for all $w\in\mathbb B$. Thus it follows from the Schwarz  lemma that
\begin{equation}\label{modulus of der}
|f'(0)|=|f'_{odd}(0)|\leq 1.
\end{equation}

Now we come to prove the last assertion in the theorem. Obviously, if $f(w)=f(0)+wf'(0)$ with $|f'(0)|=1$,  equality holds both in (\ref{regular-diam01}) and (\ref{regular-diam02}). Conversely, suppose that equality holds in (\ref{regular-diam02}), i.e. $|f'(0)|=1$. Thus $|f'_{odd}(0)|=1$, and then again by the Schwarz lemma,
$$f_{odd}(w)=wf'(0).$$
We next claim that in this case $\widetilde{d}\big(f(r\mathbb B)\big)=2r$
for each $r\in(0, 1)$. Indeed, from (\ref{g-value}) and (\ref{diameter-quo01}) it follows that for each $r\in(0, 1)$,
$$\frac{\widetilde{d}\big(f(r\mathbb B)\big)}{2r}\geq
\max_{u,v\in \overline{\mathbb B}}|g_{u,v}(0)|=\frac12 \max_{u,v\in \overline{\mathbb B}}|u-v||f'(0)|=1,$$
which together with (\ref{regular-diam10}) implies that
$$\widetilde{d}\big(f(r\mathbb B)\big)=2r$$
for each $r\in(0, 1)$, as claimed.

 Take $\xi\in\mathbb B\setminus\{0\}$ with $0<|\xi|=:r<1$ and set
 \begin{equation}\label{h-definition}
   h (w)=\frac12 \Big(f(w)-f(-\xi)\Big).
 \end{equation}
 Thus $h $ is regular on $\mathbb B$ satisfying
 $$ h (\xi)=\frac12 \Big(f(\xi)-f(-\xi)\Big)=f_{odd}(\xi)=\xi f'(0).$$
 Moreover, from the very definition (\ref{h-definition}) and Proposition \ref{regular-diameter-relation} it follows that
 \begin{equation}\label{h-condition}
\max_{|w|\leq r}|h (w)|
=\frac12 \max_{|w|\leq r}\big |f(w)-f(-\xi)\big|
\leq\frac12 {\rm{diam}}\, f(r\mathbb B)
\leq\frac12 \widetilde{d}\big(f(r\mathbb B)\big)
=r
=|h(\xi)|.
\end{equation}
Therefore, the regular function $h $ satisfies all the assumptions given in Theorem \ref{BSL}, and hence
\begin{equation}
\begin{split}
\frac{\partial |h|}{\partial \xi}(\xi)
&=\overline{\xi}\Big(h(\xi)\overline{h'(\xi)}+\big[\bar{\xi}, h(\xi)\overline{R_{\bar{\xi}}R_{\xi}h(\xi)}\,\big]+2\big[\xi, h(\xi), R_{\bar{\xi}}R_{\xi}h(\xi)\big]\Big)
\\
&>0.
\end{split}
\end{equation}
In particular,
\begin{equation}\label{Key relation01}
\begin{split}
0&=\Big\langle I_{\xi},\, \overline{\xi}\Big(h(\xi)\overline{h'(\xi)}+\big[\bar{\xi}, h(\xi)\overline{R_{\bar{\xi}}R_{\xi}h(\xi)}\,\big]+2\big[\xi, h(\xi), R_{\bar{\xi}}R_{\xi}h(\xi)\big]\Big)\Big\rangle
\\
&=\Big\langle I_{\xi}\xi,\,  h(\xi)\overline{h'(\xi)}+\big[\bar{\xi}, h(\xi)\overline{R_{\bar{\xi}}R_{\xi}h(\xi)}\,\big]+2\big[\xi, h(\xi), R_{\bar{\xi}}R_{\xi}h(\xi)\big] \Big\rangle
\\
&=\Big\langle I_{\xi}\xi,\,  h(\xi)\overline{h'(\xi)}\Big\rangle.
\end{split}
\end{equation}
Here $I_{\xi}={\rm{Im}}\,\xi/|{\rm{Im}}\,\xi|$ is the pure imaginary unit identified by $\xi$, the second equality follows from  Lemma \ref{unitary multipliers}, and the last one follows from Lemma \ref{associator} and its proof.

Substituting the values of $h(\xi)$ and $h'(\xi)$ into the preceding inequalities yields that
\begin{equation}
\begin{split}
0&=\frac1 {r^2}\Big\langle \xi I_{\xi},\,  \big(\xi f'(0)\big)\overline{f'(\xi)}\Big\rangle
\\
&=\frac1 {r^2}\Big\langle \xi I_{\xi},\, \big[\xi, f'(0), \overline{f'(\xi)}\,\big]+\xi \big(f'(0)\overline{f'(\xi)}\,\big)\Big\rangle
\\
&=\frac1 {r^2}\Big\langle \xi I_{\xi},\,\xi \big(f'(0)\overline{f'(\xi)}\,\big)\Big\rangle
\\
&=\Big\langle I_{\xi},\, f'(0)\overline{f'(\xi)}\, \Big\rangle
\\
&=-\Big\langle I_{\xi},\, f'(\xi)\overline{f'(0)}\, \Big\rangle.
\end{split}
\end{equation}
Here we have again used Lemmas \ref{unitary multipliers} and \ref{associator}. Therefore, for each $\xi\in\mathbb B\setminus\{0\}$,
\begin{equation}\label{desried result}
\begin{split}
\Big\langle I_{\xi},\, f'(\xi)\ast\overline{f'(0)}\Big\rangle
&=\sum_{n=1}^{\infty} n \Big\langle I_{\xi},\, \xi^{n-1}\big(a_n\overline{f'(0)}\,\big)\Big\rangle
\\
&=\sum_{n=1}^{\infty} n \Big\langle I_{\xi},\, \big(\xi^{n-1}a_n\big)\overline{f'(0)}-\big[\xi^{n-1}, a_n, \overline{f'(0)}\,\big]\Big\rangle
\\
&=\sum_{n=1}^{\infty} n \Big\langle I_{\xi},\, \big(\xi^{n-1}a_n\big)\overline{f'(0)}\Big\rangle
\\
&=\Big\langle I_{\xi},\, f'(\xi)\overline{f'(0)}\, \Big\rangle
\\
&=0.
\end{split}
\end{equation}
Thus by Lemma \ref{Constant lemma}, the regular function
$$\xi\mapsto f'(\xi)\ast \overline{f'(0)}$$
must be a real constant function $|f'(0)|^2=1$, and hence $f'(\xi)\equiv f'(0)$. Consequently, $f$ is of the desired form
$$f(w)=f(0)+wf'(0).$$

Now to complete the proof, it suffices to show that how equality in (\ref{regular-diam01}) for some $r_0\in(0,1)$ implies equality in (\ref{regular-diam02}). This part is completely the same as that in the proof of \cite[Theorem 3.9]{Gen-Sar} and so we omit it.
\end{proof}

\begin{proof}[Proof of Theorem $\ref{slice diam-LT}$]
The proof of this theorem is similar to that of Theorem \ref{regular diam-LT}. The only difference  is that, instead of Theorem \ref{BSL}, we use the following simple observation.
With the regular function $h$ constructed in (\ref{h-definition})  and the function $f$ in this theorem in mind, if $|f'(0)|=1$, then $f_{odd}(w)=wf'(0)$ and ${\rm{diam}}\,f(r\mathbb B_I)=2r$
for each $r\in(0, 1)$ and each $I\in\mathbb S$, and hence as in (\ref{h-condition}) we have \begin{equation}\label{h-weak-condition}
  \max_{w\in r\overline{\mathbb B}_{I_{\xi}}}|h(w)|
  =\frac12 \max_{w\in r\overline{\mathbb B}_{I_{\xi}}}\big |f(w)-f(-\xi)\big|\leq\frac12 {\rm{diam}}\, f(r\mathbb B_{I_{\xi}})=r=|h(\xi)|.
\end{equation}
Thus as in the proof of Theorem \ref{BSL}, we deduce that the directional derivative of $|h|^2$ along the direction $v_0:=I_{\xi}\xi\in T_{\xi}\big(\partial (r\mathbb B_{I_{\xi}})\big)$ at the point $\xi \in \partial (r\mathbb B_{I_{\xi}})$ vanishes, i.e.
$$\frac{\partial |h|^2}{\partial v_0}(\xi)=0.$$
This together with (\ref{der-relation2}) with $f$ replaced by $h$ and $v$ by $v_0$ implies
\begin{equation}\label{Key relation02}
\begin{split}
0&=\Big\langle I_{\xi}\xi,\,  h(\xi)\overline{h'(\xi)}+\big[\bar{\xi}, h(\xi)\overline{R_{\bar{\xi}}R_{\xi}h(\xi)}\,\big]+2\big[\xi, h(\xi), R_{\bar{\xi}}R_{\xi}h(\xi)\big] \Big\rangle
\\
&=\Big\langle I_{\xi}\xi,\,  h(\xi)\overline{h'(\xi)}\Big\rangle,
\end{split}
\end{equation}
which is $(\ref{Key relation01})$ except the first equality there and is sufficient  for obtaining $(\ref{desried result})$ and in turn the desired result. This completes the proof.
\end{proof}

\begin{proof}[Proof of Theorem $\ref{Poukka}$]
The argument is standard (cf. \cite[p.149, Theorem 9.1]{BMMPR}). Write $a_k:=f^{(k)}(0)/k!$, so that
$$f(w)=\sum_{k=0}^{\infty}w^ka_k.$$
Fix a positive integer  $n$ and a $I\in\mathbb S$, consider  the regular function on $\mathbb B$ given by
\begin{equation}\label{g-definition00}
g(w)=\sum_{k=0}^{\infty}w^k\big((1-e^{k\pi I/n})a_k\big).
\end{equation}
Notice that $g_I(z)=f_I(z)-f_I(ze^{\pi I/n})$ holds for all $z\in \mathbb B_I$. Thus together with Lemma \ref{unitary multipliers}, the absolute and locally uniform convergence of the power series in $(\ref{g-definition00})$ implies that for each $r\in(0,1)$,
\begin{equation}\label{Area-consideration00}
\begin{split}
d^2&\geq\frac{1}{2\pi}\int_{-\pi}^{\pi}\big|g(re^{I\theta})\big|^{2}d\theta
\\
&=\frac{1}{2\pi}\sum_{k, l=0}^{\infty} r^{k+l}\int_{-\pi}^{\pi}\Big\langle e^{kI\theta}\big((1-e^{k\pi I/n})a_k\big),\, e^{lI\theta}\big((1-e^{l\pi I/n})a_l\big)\Big\rangle \,d\theta
\\
&=\frac{1}{2\pi}\sum_{k, l=0}^{\infty} r^{k+l}\int_{-\pi}^{\pi}\Big\langle e^{(k-l)I\theta}\big((1-e^{k\pi I/n})a_k\big),\,  (1-e^{l\pi I/n})a_l \Big\rangle\, d\theta
\\
&=\frac{1}{2\pi}\sum_{k, l=0}^{\infty} r^{k+l}\bigg\langle \bigg(\int_{-\pi}^{\pi}e^{(k-l)I\theta}\,d\theta\bigg)\big((1-e^{k\pi I/n})a_k\big),\, (1-e^{l\pi I/n})a_l \bigg\rangle
\\
&=\sum_{k=0}^{\infty}\big|1-e^{k\pi I/n}\big|^2|a_k|^2r^{2k}.
\end{split}
\end{equation}
Thus by Lebesgue's monotone convergence theorem,
\begin{equation}\label{Area-consideration}
\sum_{k=0}^{\infty}\big|1-e^{k\pi I/n}\big|^2|a_k|^2=\lim\limits_{r\rightarrow 1^-} \sum_{k=0}^{\infty}\big|1-e^{k\pi I/n}\big|^2|a_k|^2r^{2k}\leq d^2.
\end{equation}
In particular,  $$|a_n|\leq \frac d2,$$
which is precisely inequality $(\ref{Cauchy type})$.

If equality holds in $(\ref{Cauchy type})$ for some $n_0$, then (\ref{Area-consideration}) with $n$ replaced by $n_0$ implies that
$$(1-e^{k\pi I/n_0})a_k=0$$
for all $k\neq n_0$. In particular, $a_k=0$ whenever $k$ is not a multiple of $n_0$. Thus
\begin{equation}\label{fh-relation}
f(w)=h(w^{n_0}),
\end{equation}
where
$$h(w):=\sum_{k=0}^{\infty}w^ka_{kn_0},$$
which satisfies that
\begin{equation}\label{h-definition00}
h'(0)=a_{n_0}\qquad \mbox{and} \qquad {\rm{Diam}}\, h(\mathbb B)={\rm{Diam}}\, f(\mathbb B)=d.
\end{equation}
Suppose that $d>0$. By the very definition,
$$\widehat{d}\big(h(\mathbb B)\big)\leq {\rm{diam}}\,h(\mathbb B)=d. $$
This together with Theorem \ref{slice diam-LT} implies
$$\frac{d}2=|a_{n_0}|=|h'(0)|\leq\frac12 \widehat{d}\big(h(\mathbb B)\big)\leq \frac{d}2.$$
Consequently,
$$ |h'(0)|=\frac12 \widehat{d}\big(h(\mathbb B)\big)=\frac{d}2.$$
It immediately follows from Theorem \ref{slice diam-LT} that
$$h(w)=h(0)+wh'(0),$$
which implies that $f$ is of the desired form. The proof is complete.
\end{proof}

\section{Geometric properties of octonionic slice regular functions}
\label{Geometric properties}

In this section, we use some ideas developed in the proof of Theorem \ref{BSL} to further investigate geometric properties of octonionic slice regular functions.

\subsection{The minimum principle and the open mapping theorem}

We begin with the following result,  a special case of which has been used in (\ref{desried result}).

\begin{proposition}\label{R-product wrt I-product}
Let $f$ and $g$ be two regular functions on  a symmetric slice domain $\Omega\subseteq \mathbb O$. Then the following two equalities hold:
\begin{equation}\label{inner-R-product}
\Big\langle I_w, \,f\ast g(w)\Big\rangle=\Big\langle I_w, \, f(w)g\big(f(w)^{-1}wf(w)\big)\Big\rangle,\qquad \forall\, w\in \Omega\setminus \mathcal{Z}_f,
\end{equation}
and
\begin{equation}\label{inner-R-reciprocal}
|f^{-\ast}(w)|=\frac 1{\big|f\big(f^c(w)^{-1}wf^c(w)\big)\big|}, \qquad \forall\, w\in \Omega\setminus \mathcal{Z}_{f^s}.
\end{equation}
\end{proposition}

\begin{proof}
Let $D\subseteq \mathbb R^2$ be a domain such that $w=x+yI\in \Omega$ whenever  $(x, y)\in D$ and $I\in\mathbb S$. Since $f$ and $g$ are  regular on $\Omega$, it follows from \cite[Propositions 6 and 8]{Ghiloni1} that there exist four smooth functions $\alpha, \beta, \gamma, \delta: D\rightarrow \mathbb O$   with $\alpha(x, y)=\alpha(x, -y)$, $\beta(x, -y)=-\beta(x,  y)$,  $\gamma(x, y)=\gamma(x, -y)$ and $\delta(x, -y)=-\delta(x, y)$ such that
$$
\left\{
\begin{aligned}
\frac{\partial \alpha}{\partial x}-\frac{\partial\beta}{\partial y}=0\\
\frac{\partial \alpha}{\partial y}+\frac{\partial\beta}{\partial x}=0
\end{aligned}
\right.
;\qquad \qquad
\left\{
\begin{aligned}
\frac{\partial \gamma}{\partial x}-\frac{\partial\delta}{\partial y}=0\\
\frac{\partial \gamma}{\partial y}+\frac{\partial\delta}{\partial x}=0
\end{aligned}
\right.
$$
and
\begin{equation}\label{fg-defintion}
f(x+yI)=\alpha (x, y)+I\beta (x, y), \qquad g(x+yI)=\gamma (x, y)+I\delta (x, y)
\end{equation}
for all $(x, y)\in D$ and $I\in\mathbb S$. By Remark \ref{remark on R-product02}, the regular product $f\ast g$ is defined equivalently by
\begin{equation}\label{e-def-rp}
\begin{split}
(f\ast g)(x+yI)=&\Big(\alpha(x, y)\gamma(x, y)-\beta (x, y)\delta(x, y)\Big)\\
&+I\Big(\alpha (x, y)\delta(x, y)+\beta (x, y)\gamma(x, y)\Big).
\end{split}
\end{equation}

Fix an arbitrary point $w=x_0+y_0I_0$ with $(x_0, y_0)\in D$ and $I_0\in\mathbb S$.
In the remaining argument, we will drop the coordinates $(x_0, y_0)$ from $\alpha(x_0, y_0)$, $\beta(x_0, y_0)$, $\gamma(x_0, y_0)$ and $\delta(x_0, y_0)$ for the sake of simplicity. Assume that $f(w)=\alpha  +I_0\beta\neq 0$. From Artin's theorem for alternative algebras (Theorem \ref{Artin-thm}), we know that
$$[f(w), \,I_0, \,f(w)^{-1}]=0$$
and hence $\big(f(w)I_0\big)f(w)^{-1}=f(w) \big(I_0f(w)^{-1}\big)$, which belongs to $\mathbb S$ and will be denoted by
$f(w)I_0f(w)^{-1}$ with no ambiguity. In view of (\ref{e-def-rp}),  and Lemmas \ref{unitary multipliers} and \ref{associator},
\begin{equation}\label{inner-R-product01}
\begin{split}
\Big\langle I_0, \,f\ast g(w)\Big\rangle
&=\big\langle I_0, \, \alpha \gamma+I_0(\beta\gamma)\big\rangle
+\big\langle I_0, \, I_0(\alpha \delta)-\beta\delta\big\rangle\\
&=\big\langle I_0, \, \alpha \gamma+(I_0\beta)\gamma\big\rangle
+\big\langle 1, \,  \alpha \delta\big\rangle- \big\langle I_0, \,\beta\delta\big\rangle\\
&=\big\langle I_0, \, (\alpha+I_0\beta)\gamma\big\rangle
+\big\langle 1, \,  \alpha \delta\big\rangle- \big\langle I_0, \,\beta\delta\big\rangle.
\end{split}
\end{equation}
Next we claim that
$$\Big\langle I_0, \, (\alpha+I_0\beta)\Big(\big((\alpha+I_0\beta)^{-1}I_0(\alpha+I_0\beta)\big)\delta\Big)\Big\rangle
=\big\langle 1, \,  \alpha \delta\big\rangle- \big\langle I_0, \,\beta\delta\big\rangle,$$
from which together with (\ref{inner-R-product01}) it will immediately follow   that
\begin{equation*}
\begin{split}
\Big\langle I_0, \,f\ast g(w)\Big\rangle
&=\Big\langle I_0, \, (\alpha+I_0\beta)\gamma\Big\rangle +\Big\langle I_0, \,
(\alpha+I_0\beta)\Big(\big((\alpha+I_0\beta)^{-1}I_0
(\alpha+I_0\beta)\big)\delta\Big)\Big\rangle
\\
&=\Big\langle I_0, \, f(w)\Big(\gamma+\big(f(w)^{-1}I_0
f(w)\big)\delta\Big)\Big\rangle
\\
&=\Big\langle I_0, \, f(w)g\big(f(w)^{-1}w
f(w)\big)\Big\rangle,
\end{split}
\end{equation*}
since $f(w)^{-1}wf(w)\in\mathbb S_w$ and
$g\big(f(w)^{-1}wf(w)\big)=\gamma+\big(f(w)^{-1}I_0f(w)\big)\delta$.

Thanks to (\ref{Real part free}), and Lemmas \ref{unitary multipliers} and \ref{associator}, the preceding claim can be proved as follows:
\begin{equation*}
\begin{split}
&\Big\langle I_0, \, (\alpha+I_0\beta)\Big(\big((\alpha+I_0\beta)^{-1}I(\alpha+I_0\beta)\big)\delta\Big)\Big\rangle
\\
=&\Big\langle \overline{(\alpha+I_0\beta)}I_0, \, \Big((\alpha+I_0\beta)^{-1}I_0(\alpha+I\beta)\Big)\delta \Big\rangle\\
=&|\alpha+I_0\beta|^2\Big\langle  (\alpha+I_0\beta)^{-1}I_0, \, \Big((\alpha+I_0\beta)^{-1}I_0(\alpha+I_0\beta)\Big)\delta \Big\rangle\\
=&|\alpha+I_0\beta|^2\Big\langle  (\alpha+I_0\beta)^{-1}I_0, \, \big[(\alpha+I_0\beta)^{-1}I_0,\, \alpha+I_0\beta, \, \delta\big]\\
&+\Big((\alpha+I_0\beta)^{-1}I_0\Big)\Big((\alpha+I_0\beta)\delta\Big) \Big\rangle\\
=&|\alpha+I_0\beta|^2\Big\langle  (\alpha+I_0\beta)^{-1}I_0, \,  \Big((\alpha+I_0\beta)^{-1}I_0\Big)\Big((\alpha+I_0\beta)\delta\Big) \Big\rangle\\
=&\big\langle 1, \, (\alpha+I_0\beta)\delta\big\rangle\\
=&\big\langle 1, \, \alpha\delta\big\rangle+\big\langle 1, \, (I_0\beta)\delta\big\rangle\\
=&\big\langle 1, \, \alpha\delta\big\rangle+\big\langle 1, \,I_0(\beta\delta)\big\rangle\\
=&\big\langle 1, \,  \alpha \delta\big\rangle- \big\langle I_0, \,\beta\delta\big\rangle.
\end{split}
\end{equation*}

Now the proof of equality (\ref{inner-R-product}) is complete and it remains to prove equality (\ref{inner-R-reciprocal}). To this end, we further assume that $f^s(w)\neq0$, i.e $w\in \Omega\setminus \mathcal{Z}_{f^s}$. By \cite[Corollary 19]{Ghiloni1},
\begin{equation}\label{zero-relation}
\mathcal{Z}_{f^s}=\bigcup_{u\in \mathcal{Z}_{f}}\mathbb S_u
=\bigcup_{u\in \mathcal{Z}_{f^c}}\mathbb S_u.
\end{equation}
Therefore, the restrictions $\left. f\right|_{\mathbb S_w}$ and $\left. f^c\right|_{\mathbb S_w}$ never vanish so that $f^c(w)^{-1}wf^c(w)$ makes sense and
\begin{equation}\label{w-tilde notation}
\widetilde{w}:=f^c(w)^{-1}wf^c(w)=x_0+y_0f^c(w)^{-1}I_0f^c(w)\in \mathbb S_w\subseteq \Omega\setminus \mathcal{Z}_{f^s}.
\end{equation}
Furthermore, it follows from (\ref{fg-defintion}) and Remark \ref{remark on con-sym01} that
$$f^c(x+yI)=\overline{\alpha} (x, y)+I\overline{\beta} (x, y)$$
 for all $(x, y)\in D$ and $I\in\mathbb S$, from which we deduce that
\begin{equation}\label{f-ast-notation}
\begin{split}
f^{-\ast}(w)
&=\frac 1{f^s(w)}f^c(w)\\
&=\Big(|\alpha|^2-|\beta|^2+2I_0\langle\alpha, \,\beta\rangle\Big)^{-1}
\big(\overline{\alpha} +I_0\overline{\beta}\,\big)
\end{split}
\end{equation}
so that
$$|f^{-\ast}(w)|^2=\frac{|\alpha|^2+|\beta|^2+2\langle\overline{\alpha}\beta, \,I_0\rangle}{\big(|\alpha|^2-|\beta|^2\big)^2+4\langle\alpha, \,\beta\rangle^2}.$$
Now to conclude the proof, it suffices to prove that
\begin{equation}\label{f(w-tilde)-notaion}
|f(\widetilde{w})|^2=\frac{\big(|\alpha|^2-|\beta|^2\big)^2+4\langle\alpha, \,\beta\rangle^2}{|\alpha|^2+|\beta|^2+2\langle\overline{\alpha}\beta, \,I_0\rangle}.
\end{equation}
Since
$$f(\widetilde{w})=\alpha+\big(f^c(w)^{-1}I_0f^c(w)\big)\beta
=\alpha+\Big(\big(\overline{\alpha} +I_0\overline{\beta}\,\big)^{-1}I_0
\big(\overline{\alpha} +I_0\overline{\beta}\,\big)\Big)\beta,$$
 it follows that
\begin{equation}\label{computation of norm}
\begin{split}
|f(\widetilde{w})|^2
&=|\alpha|^2+|\beta|^2
+2\Big\langle \alpha, \, \Big(\big(\overline{\alpha} +I_0\overline{\beta}\,\big)^{-1}I_0
\big(\overline{\alpha} +I_0\overline{\beta}\,\big)\Big)\beta\Big\rangle
\\
&=|\alpha|^2+|\beta|^2
+2|\alpha-\beta I_0|^{-2}\Big\langle \alpha, \, \Big(\big(\alpha-\beta I_0\big)I_0
\big(\overline{\alpha} +I_0\overline{\beta}\,\big)\Big)\beta\Big\rangle.
\end{split}
\end{equation}
We next claim that
$$\Big\langle \alpha, \, \Big(\big(\alpha-\beta I_0\big)I_0
\big(\overline{\alpha} +I_0\overline{\beta}\,\big)\Big)\beta\Big\rangle
=-\big(|\alpha|^2+|\beta|^2\big)\big\langle \overline{\alpha}\beta, \,I_0\big\rangle
+2\big\langle  \alpha, \, \beta\big\rangle^2-2|\alpha|^2|\beta|^2,$$
from which  (\ref{f(w-tilde)-notaion}) will immediately follow  and the proof will be concluded. A direct computation shows that the left-hand side of the preceding equality is exactly
\begin{equation}\label{computation of mixed term}
\begin{split}
&\Big\langle \alpha, \, \Big( \alpha I_0\overline{\alpha}+\beta I_0\overline{\beta}+ \beta\overline{\alpha}-\alpha\overline{\beta} \Big)\beta\Big\rangle
\\
=&\Big\langle \alpha, \, \alpha\big( (I_0\overline{\alpha})\beta \big)
+\big[\alpha, \,I_0\overline{\alpha}, \,\beta\big]\Big\rangle
+|\beta|^2\big\langle\alpha, \,\beta I_0\big\rangle
+\big\langle\alpha\overline{\beta}, \,\beta\overline{\alpha}\big\rangle-|\alpha|^2 |\beta|^2
\\
=&|\alpha|^2\big\langle 1, \,  \big(I_0\overline{\alpha}\big)\beta
\big\rangle+|\beta|^2\big\langle\overline{\beta}\alpha, \,I_0\big\rangle
+\Big\langle\alpha\overline{\beta}, \, 2\big\langle\alpha, \, \beta\big\rangle-\alpha\overline{\beta}\Big\rangle-|\alpha|^2 |\beta|^2
\\
=&|\alpha|^2\big\langle 1, \,I_0\big(\overline{\alpha}\beta\big)
\big\rangle-|\beta|^2\big\langle\overline{\alpha}\beta, \,I_0\big\rangle
+2\big\langle\alpha, \, \beta\big\rangle^2-2|\alpha|^2 |\beta|^2
\\
=&-\big(|\alpha|^2+|\beta|^2\big)\big\langle \overline{\alpha}\beta, \,I_0\big\rangle
+2\big\langle\alpha, \, \beta\big\rangle^2-2|\alpha|^2|\beta|^2,
\end{split}
\end{equation}
which is precisely the right-hand side as desired.
\end{proof}

Now we come to prove the following weak version of the minimum principle:
\begin{theorem}\label{WMinP}
Let $f:\Omega\rightarrow \mathbb O$ be a regular function on a symmetric slice domain $\Omega\subseteq\mathbb O$. If  $|f|$  attains a local minimum at some point $w_0\in \Omega\cap\mathbb R$, then either $f(w_0)=0$ or $f$ is constant.
\end{theorem}

We give two proofs of the theorem, which seem useful for other purposes.
\begin{proof}[The first proof of Theorem $\ref{WMinP}$]
We use a variational argument similar to the proof of Theorem $\ref{BSL}$. Without loss of generality, we may assume that $0\in \Omega$ and $w_0=0$. Suppose by contradiction that $f(0)\neq0$ or $f$ is not constant. For each $\xi\in\partial \mathbb B$, we consider the function
$$\psi_{\xi}(t):=|f(t\xi)|^2$$
defined on some interval $(-\varepsilon, \varepsilon)$ with  $\varepsilon>0$ sufficiently small. By assumption, $\psi_{\xi}(0)$ is a minimum of $\psi_{\xi}$, and hence
$$2\big\langle f(0)\overline{f'(0)}, \,\xi\big\rangle=
2\big\langle f(0), \,\xi f'(0)\big\rangle=
2\left.\Big\langle f(t\xi), \,\frac{d}{dt}f(t\xi)\Big\rangle\right|_{t=0}
=\psi_{\xi}'(0)=0$$
for all $\xi\in\partial \mathbb B$. Therefore, $f(0)\overline{f'(0)}=0$, i.e. $f'(0)=0$.
Since $f$ is not constant, there must exist a positive integer $n_0\geq2$ such that
$f^{(n_0)}(0)\neq 0$, but $f'(0)=f''(0)=\ldots=f^{(n_0-1)}(0)=0$. Thus it holds that
\begin{equation}\label{varational}
\begin{split}
\psi_{\xi}(t)&=\Big|f(0)+t^{n_0}\xi^{n_0}f^{(n_0)}(0)/n_0!+o(t^{n_0})\Big|^2 \\
&=|f(0)|^2+2t^{n_0}\Big\langle f(0)\overline{f^{(n_0)}(0)}/n_0!, \,\xi^{n_0}\Big\rangle
+o(t^{n_0})
\end{split}
\end{equation}
as $t\rightarrow 0$. Then $$\psi_{\xi}'(0)=\psi_{\xi}''(0)=\ldots=\psi_{\xi}^{(n_0-1)}(0)=0$$ and
$$\psi_{\xi}^{(n_0)}(0)=2\Big\langle f(0)\overline{f^{(n_0)}(0)}, \,\xi^{n_0}\Big\rangle.$$
If $n_0$ is even, then the minimality of $\psi_{\xi} (0)$ implies that $\psi_{\xi}^{(n_0)}(0)\geq 0$, i.e.
$$\Big\langle f(0)\overline{f^{(n_0)}(0)}, \,\xi^{n_0}\Big\rangle\geq 0$$
for all  $\xi\in\partial \mathbb B$. This is possible only if $f^{(n_0)}(0)=0$, which is a contradiction. If $n_0$ is odd, then $\psi_{\xi} (0)=0$ for all  $\xi\in\partial \mathbb B$. This also implies that $f^{(n_0)}(0)=0$, giving a contradiction as well. The proof is complete.
\end{proof}

\begin{proof}[The second proof of Theorem $\ref{WMinP}$]
Without loss of generality, we may assume that $0\in \Omega$ and $w_0=0$. We further assume that $f(0)\neq0$ and use Proposition \ref{R-product wrt I-product} to deduce that $f(0)^{-1}$ is a local maximum of $f^{-\ast}$, and thus it follows from the maximum principle (Theorem \ref{MP}) that $f^{-\ast}$ is constant and so is $f$.

First in view of Remark \ref{remark on con-sym02}, the symmetrization $f^s$ is slice preserving, i.e. $f^s(w)$ and $w$ always lie in a same complex plane and thus commute.
Then by equality (\ref{inner-R-reciprocal}) with $f$ replaced by $f^{-\ast}$ and Artin's theorem for alternative algebras (Theorem \ref{Artin-thm}), the following equality
\begin{equation}\label{f-f-inverse}
|f(w)|=\frac 1{\big|f^{-\ast}\big(f(w)^{-1}wf(w)\big)\big|}
\end{equation}
holds for all $w\in \Omega\setminus \mathcal{Z}_{f^s}$. We next consider the real differential at the point $w_0=0$ of the function given by
$$g(w):=f(w)^{-1}wf(w)=|f(w)|^{-2}\overline{f(w)}wf(w)$$
on $\Omega\setminus \mathcal{Z}_{f^s}\ni0$. Since $|f|^{-2}$ attains a local maximum  $|f(0)|^{-2}$ at $0$, its directional derivative along every direction at $0$ always vanishes.  Now a simple calculation gives that for all $v\in\partial \mathbb B$,
$$\frac{\partial g}{\partial v}(0)=|f(0)|^{-2}\overline{f(0)}vf(0)=f(0)^{-1}vf(0)\neq 0.$$
This means exactly that the real differential $dg_0$ of $g$ at $0$ is invertible. Thus  in view of the inverse mapping theorem, $g$ is a differmorphism from $B(0, r_1)$ onto $B(0, r_2)$, where $r_1$ and $r_2$ are two small positive numbers. Therefore, $f(0)^{-1}$ is a local maximum of $f^{-\ast}$ in virtue of equality (\ref{f-f-inverse}), and the desired result immediately follows.
\end{proof}

\begin{remark}\label{Remark on WMP}
Now a fairly natural question arises of whether the restriction of $w_0$ belonging to $\Omega\cap\mathbb R$ in Theorem \ref{WMinP} is superfluous. The point in the second proof of Theorem $\ref{WMinP}$ is to prove that the real differential of $g$ at the point $w=w_0$, which is the minimum point of $|f|$ such that $f(w_0)\neq0$, is non-degenerate. In the general case that $w_0\in\Omega\setminus\mathbb R$, the author does not know whether  the preceding fact necessarily holds,  and merely know that the rank of the real differential of $g$ at the point $w=w_0$ is greater than or equal to 4. If this were the case, the general minimum principle would immediately follow (Another possible approach to the general minimum principle is an  argument analogous to that in the first proof of Theorem $\ref{WMinP}$, by means of the so-called spherical power series expansion \cite[Theorem 5.4]{Ghiloni3} for octonionic slice regular functions), and in turn would imply the open mapping theorem analogous to \cite[Theorem 7.7]{GSS}. However, we have the following result, which corresponds to \cite[Theorem 7.4]{GSS}.
\end{remark}

\begin{theorem}\label{Weak Open mapping}
Let $f:\Omega\rightarrow \mathbb O$ be a nonconstant regular function on a symmetric slice domain $\Omega\subseteq \mathbb O$. If $U$ is a symmetric open subset of $\Omega$, then $f(U)$ is open. In particular, $f(\Omega)$ is open.
\end{theorem}

For each $\alpha\in\mathbb O$ and each $\delta>0$, we denote by  $V_{\alpha\delta}$ the symmetric open subset of $\mathbb O$ given by
$$V_{\alpha\delta}:=\big\{w\in \mathbb O: d(w, \mathbb S_{\alpha})<\delta \big\},$$
where $d(\cdot, \cdot)$ is the Euclidean distance function on $\mathbb O$.
For each non-identically vanishing regular function $f$ on a symmetric slice domain $\Omega\subseteq \mathbb O$, we define $\mathcal{L}_f$ to be the slice preserving function on $\Omega\setminus \mathcal{Z}_{f^s}$  given by
$$\mathcal{L}_f(w)=\frac{(f^s)'(w)}{f^s(w)},$$
which plays the role of the logarithmic derivatives of holomorphic functions of one complex variable.
Before presenting a proof of Theorem \ref{Weak Open mapping}, we first prove the following
\begin{proposition}\label{the arg-principle}
Let $f:\Omega\rightarrow \mathbb O$ be a non-identically vanishing regular function on a symmetric slice domain $\Omega\subseteq \mathbb O$. Let $\alpha \in\Omega$ and $\delta>0$ be such that $V_{\alpha\delta}\subset\subset\Omega$ and $f^s$ never vanishes on $\partial V_{\alpha\delta}$.
\begin{enumerate}
  \item [(i)] If $\alpha$ is  a zero of $f$, then the value of the following integral $$\frac1{2\pi I}\int_{\partial V_{\alpha\delta}\cap\mathbb C_I}\mathcal{L}_f(z)dz$$
 is a positive integer depending only on $\alpha, \,\delta$ and independent of $I\in\mathbb S$;
   \item [(ii)] If
 $$\frac1{2\pi I}\int_{\partial V_{\alpha\delta}\cap\mathbb C_I}\mathcal{L}_f(z)dz>0,$$
 then $f$ must have a zero on $V_{\alpha\delta}$.
\end{enumerate}
\end{proposition}

\begin{proof}
The result follows immediately from the complex argument principle applied to the slice  preserving function $f^s$ (see Remark \ref{remark on con-sym02}), together with \cite[Corollary 19]{Ghiloni1}.
\end{proof}

Now we come to prove Theorem \ref{Weak Open mapping}.

\begin{proof}[Proof of Theorem $\ref{Weak Open mapping}$]
Let $U\subseteq\Omega$ be as described. Fix an arbitrary point $\omega_0\in f(U)$. Choose one point $\alpha\in U$ with $f(\alpha)=\omega_0$, so that $f-\omega_0$ has a zero on $\mathbb S_{\alpha}\subseteq U$. Since $f$ is nonconstant, we may choose a $\delta>0$ such that $V_{\alpha\delta}\subset\subset U$ and $f(w)-\omega_0\neq 0$ for all $w\in  \overline{V}_{\alpha\delta}\setminus\mathbb S_{\alpha}$. Let $\varepsilon>0$ be such that
$$\min_{w\in \partial V_{\alpha\delta}}|f(w)-\omega_0|\geq 2\varepsilon.$$
For each $\omega\in B(\omega_0, \varepsilon)$, we have
$$\min_{w\in \partial V_{\alpha\delta}}|f(w)-\omega|\geq \min_{w\in \partial V_{\alpha\delta}}|f(w)-\omega_0|-|\omega-\omega_0|>\varepsilon. $$
This together with \cite[Corollary 19]{Ghiloni1} implies that for each $\omega\in B(\omega_0, \varepsilon)$, the symmetrization $(f-\omega)^s$ of the regular function $f-\omega$ never vanishes on the boundary $\partial V_{\alpha\delta}$ so that the following integral
$$\frac1{2\pi I}\int_{\partial V_{\alpha\delta}\cap\mathbb C_I}\mathcal{L}_{f-\omega}(z)dz$$
 is well-defined and thus determines a function of $\omega\in B(\omega_0, \varepsilon)$, depending only on $\alpha,\, \delta$ and independent of $I\in\mathbb S$. This function is obviously continuous and takes values in $\mathbb N$ by Proposition \ref{the arg-principle} (i), and hence equals identically to a positive integer, since
$$\frac1{2\pi I}\int_{\partial V_{\alpha\delta}\cap\mathbb C_I}\mathcal{L}_{f-\omega_0}(z)dz\geq 1.$$
Thus it follows from Proposition \ref{the arg-principle} (ii) that for each $\omega\in B(\omega_0, \varepsilon)$, $f-\omega$ must have a zero on $V_{\alpha\delta}$. In other words, $B(\omega_0, \varepsilon)\subseteq f(V_{\alpha\delta})\subseteq f(U)$. Since $\omega_0\in f(U)$ is arbitrarily chosen, we conclude the proof.
\end{proof}

\begin{remark}
(i) It is easy to see that Theorem \ref{Weak Open mapping} also implies  Theorem  \ref{WMinP}.

(ii) Thanks to Theorem \ref{Weak Open mapping} together with the standard slice technique, one can also prove the octonionic versions of the classical Carath\'{e}odory  and Borel-Carath\'{e}odory theorem with an approach different from and simpler than the one given in \cite{RW0}, and in turn the Bohr theorem. We leave the details to the interested reader.

\end{remark}

\subsection{The growth, distortion and covering theorems}

\begin{theorem}\label{Growth and Distortion Theorems}
Let $f$ be a  regular function on  $\mathbb B$ such that its restriction $f_I$ to $\mathbb B_I$ is injective and $f(\mathbb B_I)\subseteq \mathbb C_I $ for some $I\in \mathbb S $. If $f(0)=0$ and $f'(0)=1$, then for all  $ w\in \mathbb B$, the following inequalities hold:
\begin{eqnarray}\label{eq:111}
\qquad\ \frac{|w|}{(1+|w|)^2}\leq |f(w)|\leq \frac{|w|}{(1-|w|)^2};
\end{eqnarray}
\begin{eqnarray}\label{eq:112}
\qquad\ \frac{1-|w|}{(1+|w|)^3}\leq |f'(w)|\leq \frac{1+|w|}{(1-|w|)^3} ;
\end{eqnarray}
\begin{eqnarray}\label{eq:113}
\qquad\ \frac{1-|w|}{1+|w|}\leq \big|wf'(w)\ast f^{-\ast}(w)\big|\leq \frac{1+|w|}{1-|w|}.
\end{eqnarray}

Moreover, equality holds for one of these six inequalities at some point $w_0\in \mathbb B\setminus \{0\}$ if and only if $f$ is of the form $$f(w)=w(1-we^{I\theta})^{-\ast2} $$
with some $ \theta \in \mathbb R$.
\end{theorem}

The proof of the preceding theorem is similar to the one in \cite[Theorem 3.5]{RW2}. The only difference is that we use Proposition \ref{convex combination identity}, instead of \cite[Lemma 3.2]{RW2}. So we omit the details.

Now we  digress to the Koebe one-quarter theorem (Theorem \ref{th:Koebe-theorem}) for octonionic  regular functions  on the open unit ball $\mathbb B\subset \mathbb O$. We begin with the following simple result.

\begin{proposition}\label{open}
Let $\Omega\subseteq \mathbb O$ be a bounded domain and $f:\Omega\rightarrow \mathbb O$  a continuous function such that  $ f(\Omega)$ is open in $\mathbb O$. Let  $\alpha\in \Omega$ be a point such that
\begin{equation}\label{size-condition}
\rho:=\liminf_{w\rightarrow\partial \Omega}|f(w)-f(\alpha)|>0.
\end{equation}
Then $ B(f(\alpha), \rho)\subseteq f(\Omega)$.
\end{proposition}
\begin{proof}
For each point $\omega$ on the boundary $\partial f(\Omega)$ of $f(\Omega)$, there is a sequence $\{w_n\}_{n=1}^{\infty}$ in $\Omega$ such that
$\lim_{n\rightarrow\infty}f(w_n)=\omega$. Since $\overline{\Omega}$ is compact, we may assume that $\{w_n\}_{n=1}^{\infty}$ converges to a point, say $w_{\infty}\in \overline{\Omega}$. If $w_{\infty}\in\Omega$, then, by the continuity  of $f$, $\omega=f(w_{\infty})\in f(\Omega)$, which  contradicts the openness of $f(\Omega)$. Therefore, $w_{\infty}\in \partial \Omega$. This together with $(\ref{size-condition})$ implies  that
$$|\omega-f(\alpha)|=\lim_{n\rightarrow \infty} |f(w_n)-f(\alpha)|\geq \liminf_{w\rightarrow\partial \Omega}|f(w)-f(\alpha)|=\rho>0.$$
Therefore, the boundary $\partial f(\Omega)$ of the open set $f(\Omega)$ lies outside of the ball $B\big(f(\alpha), \rho\big)$. Consequently, $f(\Omega)$ must contain the ball $B\big(f(\alpha), \rho\big)$.
\end{proof}

\begin{proof}[Proof of Theorem $\ref{th:Koebe-theorem}$]
In view of Theorem \ref{Weak Open mapping}, $f(\mathbb B)$ is open in $\mathbb O$. Since $f(0)=0$, the desired result immediately follows from Proposition \ref{open} and the first inequality in (\ref{eq:111}).
\end{proof}

\section{A new and sharp boundary Schwarz lemma for quaternionic slice regular functions}
\label{Quaternionic BSL}

In this section, we turn our attention to quaternionic slice regular functions. In this special setting, with a completely new approach, we can   strengthen a result first proved in \cite{WR} by the author and Ren, analogous to Theorem \ref{BSL}. Our quaternionic boundary Schwarz lemma with optimal estimate involves a Lie bracket, improves considerably a well-known Osserman type estimate  and provides additionally  all the extremal functions.

\subsection{Quaternionic slice regular functions}
Let $\mathbb H$ denote the non-commutative, associative, real algebra of quaternions with standard basis $\{1,\,i,\,j, \,k\}$,  subject to the multiplication rules
$$i^2=j^2=k^2=ijk=-1.$$
Let $\langle$ , $\rangle$ denote  the standard inner product on $\mathbb H\cong\mathbb R^4$, i.e. $$\langle p,q\rangle={\rm{Re}}(p\bar{q})=\sum\limits_{n=0}^3x_ny_n$$ for any $p=x_0+x_1i+x_2j+x_3k$, $q=y_0+y_1i+y_2j+y_3k\in \mathbb H$.

We shall consider the slice regular functions defined on domains in  quaternions $\mathbb H$ with values in $\mathbb H$. To introduce the theory of quaternionic slice regular functions, we will denote by $\mathbb S$ the unit $2$-sphere of purely imaginary quaternions, i.e.
$$\mathbb S=\big\{q\in\mathbb H:q^2=-1\big\}.$$
For a given element $\xi\in\mathbb H$, we denote by $\mathbb S_{\xi}$ the associated 2-sphere (reduces to the point $\xi$ when $\xi$ is real):
$$\mathbb S_{\xi}:=\big\{q\xi q^{-1}: q\in\mathbb H\setminus\{0\}\big\}.$$
Recall that two quaternions belong to the same sphere if and only if they have the same modulus and the same real part.
For every $I \in \mathbb S $ we will denote by $\mathbb C_I$ the plane $ \mathbb R \oplus I\mathbb R $, isomorphic to $ \mathbb C$, and, if $\Omega \subseteq \mathbb H$, by $\Omega_I$ the intersection $ \Omega \cap \mathbb C_I $. Also, we will denote by $B(0, R)$ the Euclidean open  ball of radius $R$ centred at the origin, i.e.
$$B(0, R)=\big\{q \in \mathbb H:|q|<R\big\}.$$
For simplicity, we denote by $\mathbb B$ the ball $B(0, 1)$.

We can now recall the definition of slice regularity.
\begin{definition} \label{de: regular} Let $\Omega$ be a domain in $\mathbb H$. A function $f :\Omega \rightarrow \mathbb H$ is called (left) \emph{slice} \emph{regular} if, for all $ I \in \mathbb S$, its restriction $f_I$ to $\Omega_I$ is \emph{holomorphic}, i.e., it has continuous partial derivatives and satisfies
$$\bar{\partial}_I f(x+yI):=\frac{1}{2}\left(\frac{\partial}{\partial x}+I\frac{\partial}{\partial y}\right)f_I (x+yI)=0$$
for all $x+yI\in \Omega_I $.
 \end{definition}

The notions of slice domain, of symmetric slice domain and of slice derivative are similar to those already given in Section 2. Moreover, the corresponding results  hold of course for quaternionic slice regular functions, such as the splitting lemma, the representation formula, the power series expansion and so on. The regular product, regular conjugate, symmetrization and regular reciprocal of quaternionic slice regular functions can also defined in an analogous way. To have a more complete insight on the theory, we refer the reader to the monograph \cite{GSS}.

For our later purpose, we need to recall some results. The first one clarifies a nice  connection between the regular product and  the usual pointwise one (see \cite{GSS1,CGSS}):
\begin{proposition} \label{prop:RP}
Let $f$ and $g$ be  slice regular on $B=B(0,R)$. Then for all $q\in B$,
$$f\ast g(q)=
\left\{
\begin{array}{lll}
f(q)g\big(f(q)^{-1}qf(q)\big) \qquad \,\,if \qquad f(q)\neq 0;
\\
\qquad  \qquad  0\qquad  \qquad \qquad if \qquad f(q)=0.
\end{array}
\right.
$$
\end{proposition}

The second one shows that the regular quotient is nicely related to the pointwise quotient (see \cite{Stop1, Stop2}):
\begin{proposition} \label{prop:Quotient Relation}
Let $f$ and $g$ be slice regular  on  $B=B(0,R)$. Then for all $q\in B \setminus \mathcal{Z}_{f^s}$, $$f^{-\ast}\ast g(q)=f\big(T_f(q)\big)^{-1}g\big(T_f(q)\big),$$
where $T_f:B \setminus \mathcal{Z}_{f^s}\rightarrow B \setminus \mathcal{Z}_{f^s}$ is defined by
$T_f(q)=f^c(q)^{-1}qf^c(q)$. Furthermore, $T_f$ and $T_{f^c}$ are mutual inverses so that $T_f$ is a diffeomorphism.
\end{proposition}

For any two numbers $x_0$, $y_0\in\mathbb R$ and each $R>0$, we denote by $U(x_0+y_0\mathbb S, R)$
the symmetric open subset of $\mathbb H$ given by
$$U(x_0+y_0\mathbb S, R):=\big\{q\in\mathbb H:\big|(q-x_0)^2+y_0^2\big|<R^2\big\}.$$
 The third one was the so-called spherical series expansion proved in \cite{Stop3} for slice regular functions; see Theorems 4.1 and 6.1 there for more details.
\begin{theorem}\label{direction derivative}
Let $f$ be a slice regular function on a symmetric slice domain $\Omega$, and let $q_0=x_0+Iy_0\in U(x_0+y_0\mathbb S, R)\subseteq\Omega$. Then there exists $\{A_n\}_{n\in\mathbb N}\subset \mathbb H$ such that
 \begin{equation}\label{150}
 f(q)=\sum\limits_{n=0}^{\infty}
 \big((q-x_0)^2+y_0^2\big)^n\big(A_{2n}+(q-q_0)A_{2n+1}\big)
 \end{equation}
 for all $q\in U(x_0+y_0\mathbb S, R).$
 \end{theorem}

 As a consequence of Theorem \ref{direction derivative}, we obtain that for all $v\in\mathbb H$ with $|v|=1$ the directional derivative of $f$ along $v$ at a point $q_0$ is given by
$$\frac{\partial f}{\partial v}(q_0)=\lim\limits_{t\rightarrow 0}\frac{f(q_0+tv)-f(q_0)}{t}=vA_1+(q_0v-v\overline{q}_0)A_2,$$
where $$A_1=R_{q_0}f(\overline{q}_0)=\partial_s f(q_0)  \qquad \mbox{and} \qquad A_2=R_{\overline{q}_0}R_{q_0}f(q_0)$$
are obtained in the same manner as in (\ref{de:Rf1}) and (\ref{de:Rf2}).
In particular, there holds that
$$f'(q_0)=R_{q_0}f(q_0)=A_1+2\,{\rm{Im}}(q_0)A_2.$$

\subsection{Formulation  and proof of quaternionic boundary Schwarz lemma}

Our subsequent argument involves  the so-called slice regular M\"{o}bius transformations of $\mathbb B$ onto $\mathbb B$, which are slice regular functions  $f$ on $\mathbb B$ given by
$$f(q)=\big(1-q\overline{u}\big)^{-\ast}\ast\big(q-u\big)v$$ with $u\in \mathbb B$ and $v\in \partial \mathbb B$ (see \cite[Corollary 7.2]{Stop4}; also \cite[Corollary 9.17]{GSS}). It is also useful to recall  the quaternionic version of the classical Julia lemma (see \cite[Theorem 1]{WR}):
\begin{theorem}\label{Julia}
Let $f$ be a  slice  regular self-mapping of the open unit ball $\mathbb B$ and let $\xi\in\partial \mathbb B$. Suppose that there exists a sequence $\{q_n\}_{n\in \mathbb N}\subset \mathbb B$ converging to $\xi$ as $n$ tends to $\infty$, such that the limits
$$\alpha:=\lim\limits_{n\rightarrow\infty}\frac{1-|f(q_n)|}{1-|q_n|}$$
and
$$\eta:=\lim\limits_{n\rightarrow\infty}f(q_n)$$
exist $($finitely$)$. Then $\alpha>0$ and the  inequality
\begin{equation}\label{eq:11}
{\rm{Re}}\Big(\big(1-f(q)\overline{\eta}\big)^{-\ast}\ast\big(1+f(q)\overline{\eta}\big)\Big)
\geq \frac{1}{\alpha}\
{\rm{Re}}\Big(\big(1-q\overline{\xi}\,\big)^{-\ast}\ast\big(1+q\overline{\xi}\,\big)\Big)
\end{equation}
holds throughout the open unit ball $\mathbb B$ and is strict except for slice regular M\"obius transformations of $\mathbb B$.
\end{theorem}

Inequality $(\ref{eq:11})$ will be called  Julia's inequality for the convenience of referring back to it.

Now we state and prove the main result of this section. The proof is  based on Theorem \ref{Julia}, instead of a Lindel\"{o}f type inequality proved in \cite[Propostion 3]{WR}.
\begin{theorem}\label{Generalized Herzig}
Let $\xi\in \partial\mathbb B$ and $f$ be a slice regular function on $\mathbb B\cup\mathbb S_{\xi}$ such that $f(\mathbb B)\subseteq\mathbb B$ and $f(\xi)\in \partial\mathbb B$. Denote by $\delta$ the quantity
\begin{equation*}\label{delta-notation}
 \overline{\xi}\Big(f(\xi)\overline{f'(\xi)}+\big[\bar{\xi}, f(\xi)\overline{R_{\bar{\xi}}R_{\xi}f(\xi)}\,\big]\Big).
\end{equation*}
Then
\begin{enumerate}
\item[(i)]
the following sharp estimate  holds:
\begin{equation}\label{Schwarz ineq}
\delta \geq \dfrac{2}
{\mathcal{S}+\dfrac{1-|f(0)|^2}{|f(\xi)-f(0)|^2}},
\end{equation}
where
\begin{equation}\label{S-notation}
\mathcal{S}:=
{\rm{Re}}\Big(f'(0)\big(f(\xi)-f(0)\big)^{-1}\xi\big(1-f(0)\overline{f(\xi)}\,\big)^{-1}\Big),
\end{equation}
and
$$R_{\xi}f(q):=(q-\xi)^{-\ast}\ast\big(f(q)-f(\xi)\big).$$
Equality in inequality $(\ref{Schwarz ineq})$ holds if and only if $f$ is of the form
\begin{equation}\label{S-exe-funs}
f(q)=\Big(1-q\big(1-qa\bar{\eta}\big)^{-\ast}\ast
\big(q\bar{\eta}-a\big)\overline{f(0)v}\Big)^{-\ast}
\ast\Big(f(0)-q\big(1-qa\bar{\eta}\big)^{-\ast}\ast\big(q\bar{\eta}-a\big)\bar{v}\Big),
\end{equation}
 where $$a\in [-1,1),\qquad v=\big(f(0)-f(\xi)\big)^{-1}\xi \big(1-f(\xi)\overline{f(0)}\,\big)\in\partial\mathbb B,$$
  and
  $$\eta=\big(1-f(\xi)\overline{f(0)}\,\big)^{-1}\xi
  \big(1-f(\xi)\overline{f(0)}\,\big)\in\partial\mathbb B.$$
Moreover, it holds that
\begin{equation}\label{inner estimate}
 \Big\langle f(t\xi),\, f(\xi)\Big\rangle\geq \frac{(\delta+1)t-(\delta-1)}{(\delta+1)-(\delta-1)t}, \qquad \forall\, t\in(-1,1),
\end{equation}
with equality for some $t_0\in(-1,1)$ if and only if
\begin{equation}\label{inner extremal funs}
f(q)=\Big(q(\delta-1)-\xi(\delta+1)\Big)^{-\ast}\ast \Big(\xi(\delta-1)- q(\delta+1)\Big)f(\xi).
\end{equation}

\item[(ii)]
if further  $$f^{(k)}(0)=0, \qquad \forall\ k=0,1,\ldots,n-1$$
for some  $n\in\mathbb N$, then
 $$\delta
 \geq n+\dfrac{2}
 {\mathcal{T}+\dfrac{1-|f^{(n)}(0)/n!|^2}{\big|f(\xi)-\xi^nf^{(n)}(0)/n!\big|^2}},$$
 where
$$\mathcal{T}:={\rm{Re}}\bigg(\frac{f^{(n+1)}(0)}{(n+1)!}
\Big(\xi^{-n}f(\xi)-f^{(n)}(0)/n!\Big)
\xi\Big(1-f^{(n)}(0)\overline{\xi^{-n}f(\xi)}/n!\Big)^{-1}\bigg).$$
Equality holds for the last  inequality if and only if $f$ is of the form  $$f(q)=q^n\bigg(1-q\big(1-qb\bar{\eta}\big)^{-\ast}\ast
\big(q\bar{\eta}-b\big)\frac{\overline{f^{(n)}(0)v}}{n!}\bigg)^{-\ast}
\ast\bigg(\frac{f^{(n)}(0)}{n!}-q\big(1-qb\bar{\eta}\big)^{-\ast}
\ast\big(q\bar{\eta}-b\big)\bar{v}\bigg),$$
 where $$b\in [-1,1),\qquad v=\Big(\xi^nf^{(n)}(0)/n!-f(\xi)\big)^{-1}\xi \big(\xi^n-f(\xi)\overline{f^{(n)}(0)}/n!\Big)\in\partial\mathbb B,$$
  and
  $$\eta=\big(\xi^n-f(\xi)\overline{f^{(n)}(0)}/n!\big)^{-1}\xi
  \big(\xi^n-f(\xi)\overline{f^{(n)}(0)}/n!\,\big)\in\partial\mathbb B.$$
 In particular, $$\overline{\xi}\Big(f(\xi)\overline{f'(\xi)}+\big[\bar{\xi}, f(\xi)\overline{R_{\bar{\xi}}R_{\xi}f(\xi)}\,\big]\Big)
 > n$$ unless $f(q)=q^nu$ for some $u\in\partial\mathbb B$.
\\
Moreover, it holds that
\begin{equation*}
 \Big\langle f(t\xi),\, f(\xi)\Big\rangle\geq t^n\frac{(\delta-n+1)t-(\delta-n-1)}
 {(\delta-n+1)-(\delta-n-1)t}, \qquad \forall\, t\in(-1,1),
\end{equation*}
with equality for some $t_0\in(-1,1)$ if and only if
\begin{equation*}
f(q)=q^n\Big(q(\delta-n-1)-\xi(\delta-n+1)\Big)^{-\ast}\ast \Big(\xi(\delta-n-1)- q(\delta-n+1)\Big)\overline{\xi}^{\,n}f(\xi).
\end{equation*}
\end{enumerate}
\end{theorem}

\begin{proof}[Proof of Theorem $\ref{Generalized Herzig}$]
We first prove the assertion $\textrm{(i)}$. In \cite[Theorem 4]{WR},  we have proved that
\begin{equation}\label{15110}
\frac{\partial |f|}{\partial \xi}(\xi)=\overline{\xi}\Big(f(\xi)\overline{f'(\xi)}+\big[\bar{\xi}, f(\xi)\overline{R_{\bar{\xi}}R_{\xi}f(\xi)}\,\big]\Big).
\end{equation}
So to obtain the desired sharp estimate in (\ref{Schwarz ineq}), it suffices to prove that
\begin{equation}\label{direction-der-estimate}
\frac{\partial |f|^2}{\partial \xi}(\xi)\geq \dfrac{4}
{\mathcal{S}+\dfrac{1-|f(0)|^2}{|f(\xi)-f(0)|^2}}
\end{equation}
with $\mathcal{S}$ being the same as in (\ref{S-notation}),
we proceed as follows.
Set
\begin{equation}\label{v-notation}
 v=\big(f(0)-f(\xi)\big)^{-1}\xi \big(1-f(\xi)\overline{f(0)}\,\big),
\end{equation}
which belongs to $\partial\mathbb B$, for $f(\xi)\in\partial\mathbb B$ by assumption. Set
\begin{equation}\label{15111}
g(q):=\big(1-f(q)\overline{f(0)}\,\big)^{-\ast}\ast\big(f(0)-f(q)\big)v,
\end{equation}
 then $g$ is a slice regular function on $\mathbb B\cup\mathbb S_{\xi}$ such that $g(\mathbb B)\subseteq\mathbb B$. Furthermore, it is evident that $g(0)=0$ and
\begin{equation}\label{15112}
g'(0)=-\frac{f'(0)}{1-|f(0)|^2}v.
\end{equation}
Denote
\begin{equation}\label{eta-notation}
\eta=T_{1-f(0)\ast f^c}(\xi)\in\mathbb \partial \mathbb B,
\end{equation}
which is a boundary fixed point of $g$. Indeed, it easily follows from Proposition \ref{prop:Quotient Relation}, (\ref{v-notation}) and (\ref{15111}) that
\begin{equation}\label{eta-fixed}
  g(\eta)=\big(1-f(\xi)\overline{f(0)}\,\big)^{-1}\xi\big(1-f(\xi)\overline{f(0)}\,\big)
=T_{1-f(0)\ast f^c}(\xi)=\eta,
\end{equation}
and hence the slice regular function $g$ satisfies all the assumptions in Theorem \ref{Generalized Herzig}.

We next claim that
\begin{equation}\label{key equation}
\frac{\partial |f|^2}{\partial \xi}(\xi)
=\frac{\big|f(0)-f(\xi)\big|^2}{1-|f(0)|^2}\lim\limits_{t\rightarrow 0^+}\frac{1-\big|g \circ T_{1-f(0)\ast f^c}(\xi-t\xi)\big|^2}{t}.
\end{equation}
First, from (\ref{15111}) we obtain that
$$f(q)=\big(1-g(q)\bar{v}\overline{f(0)}\,\big)^{-\ast}\ast\big(f(0)-g(q)\bar{v}\big).$$
This together with Proposition \ref{prop:Quotient Relation} implies
\begin{equation}\label{15113-01}
f(q)=\Big(1-g\circ T_{1-g\overline{f(0)v}}(q)\overline{f(0)v}\Big)^{-1}\Big(f(0)-g\circ T_{1-g\overline{f(0)v}}(q)\bar{v}\Big),
\end{equation}
from which one easily deduces  that
$$1-|f(q)|^2=\frac{\big(1-|f(0)|^2\big)\big(1-\big|g\circ T_{1-g\overline{f(0)v}}(q)\big|^2\big)}{\big|1-g\circ T_{1-g\overline{f(0)v}}(q)\overline{f(0)v}\big|^2}.$$
Consequently,
\begin{equation}\label{D-derivative}
\begin{split}
\frac{\partial |f|^2}{\partial \xi}(\xi)
&=\lim\limits_{t\rightarrow 0^+}\frac{1-\big|f(\xi-t\xi)\big|^2}{t}\\
&=\frac{1-|f(0)|^2}{\big|f(0)-g\circ T_{1-g\overline{f(0)v}}(\xi)\bar{v}\big|^2}\lim\limits_{t\rightarrow 0^+}\frac{1-\big|g\circ T_{1-g\overline{f(0)v}}(\xi-t\xi)\big|^2}{t}.\\
\end{split}
\end{equation}
Now a direct calculation gives that
$$1-g\overline{f(0)v}=\big(1-|f(0)|^2\big)\big(1-f\overline{f(0)}\,\big)^{-\ast},$$
which leads to
\begin{equation}\label{fg-translation}
T_{1-g\overline{f(0)v}}
=T_{(1-f\overline{f(0)})^{-\ast}}
=T_{1-f(0)\ast f^c}.
\end{equation}
This fact together with the notation of $\eta$ in (\ref{eta-notation}) implies that
\begin{equation}\label{T-relation}
\eta=T_{1-f(0)\ast f^c}(\xi)=T_{1-g\overline{f(0)v}}(\xi).
\end{equation}
Furthermore, it follows from (\ref{15113-01}) and (\ref{T-relation}) that
$$g\circ T_{1-g\overline{f(0)v}}(\xi)\bar{v}=g(\eta)\bar{v}=\eta\bar{v}
=\big(1-f(\xi)\overline{f(0)}\,\big)^{-1}\xi\big(f(0)-f(\xi)\big)$$
and hence
\begin{equation}\label{fg-modulus}
\big|f(0)-g\circ T_{1-g\overline{f(0)v}}(\xi)\bar{v}\big|
=\frac{1-|f(0)|^2}{\big|f(0)-f(\xi)\big|}.
\end{equation}
Now  (\ref{key equation}) immediately follows by substituting  (\ref{fg-translation}) and (\ref{fg-modulus}) into  (\ref{D-derivative}).

Next we turn to the estimate from below of the limit
$$\lim\limits_{t\rightarrow 0^+}\frac{1-\big|g \circ T_{1-f(0)\ast f^c}(\xi-t\xi)\big|^2}{t}$$
appeared in (\ref{key equation}). At first sight,
it should be  the directional derivative of $|g|^2$ along $\eta$ at the boundary point $\eta\in\mathbb \partial \mathbb B$. Unfortunately, it is in general not the case (It is obviously the case for $\xi=1$ or $f(0)=0$).
Even though the smooth curve
$$t\mapsto\Gamma(t):=T_{1-g\overline{f(0)v}}(\xi-t\xi)$$
defined on some interval $(-\varepsilon, \varepsilon)$ with  $\varepsilon>0$ sufficiently small goes through the point
$$\Gamma(0)=T_{1-g\overline{f(0)v}}(\xi)=\eta\in\mathbb \partial \mathbb B,$$ its tangent vector $\Gamma'(0)$ at $t=0$ is not necessarily the same as the direction $\eta\in\mathbb \partial \mathbb B$.  However, we still can  estimate the above limit  in virtue of Theorem \ref{Julia}. Indeed, applying Theorem \ref{Julia} and Julia inequality (\ref{eq:11}) to the slice regular function $h(q):=q^{-1}g(q)$ mapping  $\mathbb B$  to $\overline{\mathbb B}$ with $h(\bar{\eta})=1$ yields that

\begin{equation*}\label{estimate-below of $h$}
\lim\limits_{t\rightarrow 0^+}\frac{1-\big|h \circ T_{1-f(0)\ast f^c}(\xi-t\xi)\big|^2}{t}
\geq \frac1{{\rm{Re}}\Big(\big(1-h(0)\big)^{-1}\big(1+h(0)\big)\Big)}\\
=\frac{|1-h(0)|^2}{1-|h(0)|^2},
\end{equation*}
and hence
\begin{equation}\label{estimate-below of $g$}
\begin{split}
\lim\limits_{t\rightarrow 0^+}\frac{1-\big|g \circ T_{1-f(0)\ast f^c}(\xi-t\xi)\big|^2}{t}
&=1+\lim\limits_{t\rightarrow 0^+}\frac{1-\big|h \circ T_{1-f(0)\ast f^c}(\xi-t\xi)\big|^2}{t}\\
&\geq \frac{2\big(1-\textrm{Re}\,h(0)\big)}{1-|h(0)|^2}\\
&\geq \frac{2}{1+\textrm{Re}\,h(0)}\\
&=\frac{2}{1+\textrm{Re}\,g'(0)}.
\end{split}
\end{equation}
Now substituting  (\ref{v-notation}), (\ref{15112}) and (\ref{estimate-below of $g$}) into (\ref{key equation})  yields the desired sharp estimate in (\ref{direction-der-estimate}), and thus  completes the proof of inequality (\ref{Schwarz ineq}).

If equality holds for inequality in (\ref{Schwarz ineq}), then equalities hold for all the inequalities in (\ref{estimate-below of $g$}), thus from the condition for equality in the Julia inequality (\ref{eq:11}) and the above deduction of (\ref{estimate-below of $g$}) it follows that $h$ is of the form
\begin{equation}\label{h-exe-funs}
h(q)=\big(1-qa\bar{\eta}\big)^{-\ast}\ast\big(q\bar{\eta}-a\big)
\end{equation}
with some constant $a\in [-1,1)$.
Consequently, $f$ must be of the form
\begin{equation}\label{exe-funs}
f(q)=\Big(1-q\big(1-qa\bar{\eta}\big)^{-\ast}\ast\big(q\bar{\eta}-a\big)\overline{f(0)v}\Big)^{-\ast}
\ast\Big(f(0)-q\big(1-qa\bar{\eta}\big)^{-\ast}\ast\big(q\bar{\eta}-a\big)\bar{v}\Big),
\end{equation}
where $a\in [-1,1)$, and $v$ and $\eta$ are the same as those in (\ref{v-notation}) and (\ref{eta-notation}), respectively.
Therefore, the equality in inequality (\ref{Schwarz ineq}) can hold only for slice regular self-mappings  of  the form (\ref{exe-funs}), and a direct calculation shows that it does indeed hold for all such slice regular self-mappings. Now to complete the proof of $\textrm{(i)}$, it remains to prove inequality  $(\ref{inner estimate})$. To this end, we use the splitting lemma (cf. \cite[Lemma 1.3]{GSS}). Let $I\in\mathbb S$ be such that $\xi\in \partial\mathbb B\cap \mathbb C_I$ and let us split the slice regular function $f\overline{f(\xi)}$ as
$$f(z)\overline{f(\xi)}=\varphi(z)+\psi(z)J,\qquad \forall\, z\in\mathbb B_I,$$
where $J\in\mathbb S$ and $J\perp I$, and $\varphi$, $\psi$ are two holomorphic self-mappings of $\mathbb B_I$ satisfying
\begin{equation}\label{2-norm}
|f (z)|^2=|\varphi(z)|^2+|\psi(z)|^2
\end{equation}
for all $z\in\mathbb B_I$.
Moreover, it is evident that $$\varphi(\xi)=1, \qquad \psi(\xi)=0,$$
and
$$\Big\langle f(t\xi), f(\xi)\Big\rangle={\rm{Re}}\Big(f(t\xi) \overline{f(\xi)}\Big)={\rm{Re}}\,\varphi(t\xi).$$
Now inequality  $(\ref{inner estimate})$ follows immediately by applying Minda's theorem (see \cite[p. 135, Theorem 1]{Minda}) to the holomorphic self-mapping $\varphi$ of $\mathbb B_I$ and noticing that
$$\delta=\frac{\partial |f|}{\partial \xi}(\xi)=\frac{\partial |\varphi|}{\partial \xi}(\xi)=\xi\varphi'(\xi).$$
Here the last equality follows directly from an elementary geometric consideration about $\varphi$ at the boundary point $\xi$ or alternatively from the classical Julia-Wolff-Carath\'{e}odory theorem (cf. \cite{Sarason1}; also \cite[p. 48 (VI--3)]{Sarason2}).

If equality holds for inequality $(\ref{inner estimate})$ at some $t_0\in(-1,1)$, then it again follows from Minda's theorem that
\begin{equation}\label{Minda-Expression}
\varphi(z)=\frac{(\delta-1)\xi-(\delta+1)z}{(\delta-1)z-(\delta+1)\xi},\qquad \forall\, z\in\mathbb B_I.
\end{equation}
Furthermore, it follows from equality in (\ref{2-norm}) that
$$|\psi(z)|^2=|f (z)|^2-|\varphi(z)|^2\leq 1-|\varphi(z)|^2, \qquad \forall\, z\in\mathbb B_I,$$
which together with (\ref{Minda-Expression}) implies that $\psi\equiv0$, in virtue of the maximum principle,
and hence $f$  must be of the from in $(\ref{inner extremal funs})$.
 This completes the proof of $\textrm{(i)}$ and it remains to prove $\textrm{(ii)}$.

However, $\textrm{(ii)}$ follows easily from $\textrm{(i)}$ by considering  the slice regular function $h(q):=q^{-n}f(q)$ and noticing  that
 $$h(0)=\frac{f^{(n)}(0)}{n!},\qquad h'(0)=\frac{f^{(n+1)}(0)}{(n+1)!}.$$
Moreover,
$$f(\xi)\overline{f'(\xi)}+\big[\bar{\xi}, f(\xi)\overline{R_{\bar{\xi}}R_{\xi}f(\xi)}\,\big]
=n\xi+h(\xi)\overline{h'(\xi)}+\big[\bar{\xi}, h(\xi)\overline{R_{\bar{\xi}}R_{\xi}h(\xi)}\,\big]$$
as one easily verifies. Now the proof is complete.
\end{proof}

\begin{remark}\label{rem-Schwarz4}
In the preceding proof of the desired sharp estimate in (\ref{Schwarz ineq}),   the second inequality in (\ref{estimate-below of $g$}) plays a key role. If we make  use of the first one in (\ref{estimate-below of $g$}), we will obtain more precise estimate than that in (\ref{Schwarz ineq}). Formally, this estimate will be  much more complicated, but it is the same as  that in (\ref{Schwarz ineq}) if the functions of concern are the extremal functions given in (\ref{S-exe-funs}).
\end{remark}

\begin{remark}\label{rem-Schwarz5}
From inequality (\ref{inner estimate}), we can obtain the following estimate:
\begin{equation}\label{inner of ders}
\big\langle \xi^2f''(\xi), \, f(\xi)\big\rangle\geq \delta(\delta-1).
\end{equation}
Indeed, from inequality (\ref{inner estimate}) and the notion of $\delta$ it follows that
\begin{equation*}
\begin{split}
\big\langle  f(t\xi)-f(\xi)-(t-1)\xi f'(\xi), \, f(\xi)\big\rangle
&=\big\langle  f(t\xi), \, f(\xi)\big\rangle-1-(t-1)\delta\\
&\geq \frac{(\delta+1)t-(\delta-1)}{(\delta+1)-(\delta-1)t}-1-(t-1)\delta\\
&=(t-1)^2\frac{\delta(\delta-1)}{(\delta+1)-(\delta-1)t}
\end{split}
\end{equation*}
for all $t\in(-1,1)$. Now dividing by $(t-1)^2$ on both sides and then letting $t\rightarrow 1^-$ yields (\ref{inner of ders}). Alternatively, (\ref{inner of ders}) can also be proved by an argument  by means of the convexity of $f(\mathbb B)$ at the point $\xi\in\partial \mathbb B$. This argument seems more natural in principle, but rather difficult to deal with in practise, because of the computation of the second order differential of $f$.

Conversely, inequality (\ref{inner of ders}) reveals in a certain sense the convexity of the image $f(\mathbb B)$ at the point $\xi\in\partial \mathbb B$. For simplicity, we further assume that the slice regular function $f$ in Theorem \ref{Generalized Herzig} maps $\mathbb B_I$ into itself, i.e. $f(\mathbb B_I)\subseteq\mathbb B_I$, where $I=I_{\xi}$ is the pure imaginary unit identified by $\xi\in\partial\mathbb B$. In this special case, $\delta$ is precisely the positive number $\xi f'(\xi)/f(\xi)$, and inequality (\ref{inner of ders}) becomes
$${\rm{Re}}\bigg(\frac{\xi^2f''(\xi)}{f(\xi)}\bigg)\geq \delta(\delta-1).$$
We then obtain that
\begin{equation}\label{Convexity ineq}
{\rm{Re}}\bigg(\frac{\xi f''(\xi)}{f'(\xi)}+1\bigg)\geq \delta>0,
\end{equation}
which together with a well-known analytical characterization of convexity (cf. \cite[Theorem 2.2.3]{GG}) implies that $f(\mathbb B_I)$ is convex at $\xi\in\partial\mathbb B$. Furthermore, what is more interesting is that as shown by Theorem \ref{Generalized Herzig} (ii)  and inequality (\ref{Convexity ineq}), the higher the vanishing order of $f$ at the origin $0$ is, the more convex at the boundary point $\xi\in\partial\mathbb B$ the image $f(\mathbb B_I)$ of $\mathbb B_I$ under $f$ is, i.e. the bigger the number $${\rm{Re}}\bigg(\frac{\xi f''(\xi)}{f'(\xi)}+1\bigg)$$
is. Intuitively, this is indeed  the case.

\end{remark}

\subsection{Some corollaries of Theorem \ref{Generalized Herzig}}

First notice that  the term  on the right-hand side of  inequality (\ref{Schwarz ineq}) is clearly positive, for
$$|f'(0)|\leq 1-|f(0)|^2$$
 as shown by the Schwarz-Pick lemma (see \cite{BS, SABC}). Replacing the real part in the notation of $\mathcal{S}$ appearing in inequality (\ref{Schwarz ineq}) by  modulus yields inequality (\ref{weak-Schwarz ineq}) below. Hence,  the following corollary is a weaker version of Theorem \ref{Generalized Herzig}.

\begin{corollary}\label{Cor-Generalized Herzig}
Let $\xi\in \partial\mathbb B$ and $f$ be a slice regular function on $\mathbb B\cup\mathbb S_{\xi}$ such that $f(\mathbb B)\subseteq\mathbb B$ and $f(\xi)\in \partial\mathbb B$. Then
\begin{enumerate}
\item[(i)]
the following sharp estimate  holds:
\begin{equation}\label{weak-Schwarz ineq}
\overline{\xi}\Big(f(\xi)\overline{f'(\xi)}+\big[\bar{\xi}, f(\xi)\overline{R_{\bar{\xi}}R_{\xi}f(\xi)}\,\big]\Big)
\geq \frac{2\big|f(\xi)-f(0)\big|^2}{1-|f(0)|^2+|f'(0)|}.
\end{equation}
Moreover, equality holds for the last  inequality if and only if $f$ is of the form
\begin{equation}\label{functional equation}
f(q)=\Big(1-q\big(1-qa\bar{\eta}\big)^{-\ast}\ast
\big(q\bar{\eta}-a\big)\overline{f(0)v}\Big)^{-\ast}
\ast\Big(f(0)-q\big(1-qa\bar{\eta}\big)^{-\ast}\ast\big(q\bar{\eta}-a\big)\bar{v}\Big),
\end{equation}
 where $$a\in [-1,0],\qquad v=\big(f(0)-f(\xi)\big)^{-1}\xi \big(1-f(\xi)\overline{f(0)}\,\big)\in\partial\mathbb B,$$
  and
  $$\eta=\big(1-f(\xi)\overline{f(0)}\,\big)^{-1}\xi
  \big(1-f(\xi)\overline{f(0)}\,\big)\in\partial\mathbb B.$$

\item[(ii)]
if further  $$f^{(k)}(0)=0, \qquad \forall\ k=0,1,\ldots,n-1$$
for some  $n\in\mathbb N$, then
 $$\overline{\xi}\Big(f(\xi)\overline{f'(\xi)}+\big[\bar{\xi}, f(\xi)\overline{R_{\bar{\xi}}R_{\xi}f(\xi)}\,\big]\Big)
 \geq n+\frac{2\big|f(\xi)-\xi^nf^{(n)}(0)/n!\big|^2}{1-\big|f^{(n)}(0)/n!\big|^2
 +\big|f^{(n+1)}(0)\big|/(n+1)!}.$$
Moreover, equality holds for the last  inequality if and only if $f$ is of the form  $$f(q)=q^n\bigg(1-q\big(1-qb\bar{\eta}\big)^{-\ast}\ast
\big(q\bar{\eta}-b\big)\frac{\overline{f^{(n)}(0)v}}{n!}\bigg)^{-\ast}
\ast\bigg(\frac{f^{(n)}(0)}{n!}-q\big(1-qb\bar{\eta}\big)^{-\ast}
\ast\big(q\bar{\eta}-b\big)\bar{v}\bigg),$$
 where $$b\in [-1,0],\qquad v=\Big(\xi^nf^{(n)}(0)/n!-f(\xi)\Big)^{-1}\xi \Big(\xi^n-f(\xi)\overline{f^{(n)}(0)}/n!\Big)\in\partial\mathbb B,$$
  and
  $$\eta=\big(\xi^n-f(\xi)\overline{f^{(n)}(0)}/n!\big)^{-1}\xi
  \big(\xi^n-f(\xi)\overline{f^{(n)}(0)}/n!\,\big)\in\partial\mathbb B.$$
\end{enumerate}
\end{corollary}

\begin{proof}
We only give a proof of the assertion $\textrm{(i)}$, the other one being similar. Inequality (\ref{weak-Schwarz ineq})  follows immediately by replacing the real part in the notation of $\mathcal{S}$ appearing in inequality (\ref{Schwarz ineq}) by  modulus, and equality in (\ref{weak-Schwarz ineq}) holds if and only if
$$f'(0)\big(f(\xi)-f(0)\big)^{-1}
\xi\big(1-f(0)\overline{f(\xi)}\,\big)^{-1}\in\mathbb R^+,$$
which is equivalent to $h(0)\in \mathbb [0, 1]$, i.e. $a\in [-1,0]$.
Here the function $h$ is the one in (\ref{h-exe-funs}).
\end{proof}

Clearly, Corollary \ref{Cor-Generalized Herzig} implies \cite[Theorem 4]{WR}:
$$\overline{\xi}\Big(f(\xi)\overline{f'(\xi)}+\big[\bar{\xi}, f(\xi)\overline{R_{\bar{\xi}}R_{\xi}f(\xi)}\,\big]\Big)
\geq \frac{2\big(1-|f(0)|\big)^2}{1-|f(0)|^2+|f'(0)|},$$
and provides additionally all the extremal functions:
$$f(q)=\Big(1+q\big(1-qa\bar{\xi}\,\big)^{-\ast}
\ast\big(q\bar{\xi}-a\big)\bar{\xi}|f(0)|\Big)^{-\ast}
\ast\Big(|f(0)|+q\big(1-qa\bar{\xi}\,\big)^{-\ast}
\ast\big(q\bar{\xi}-a\big)\bar{\xi}\Big)f(\xi)$$
with $a\in[-1, 0]$. If the slice regular function $f$ considered in Theorem \ref{Generalized Herzig} has the interior fixed point $0$ and a boundary fixed point $\xi\in\partial \mathbb B$, then  Theorem \ref{Generalized Herzig} (i) implies:

\begin{corollary}\label{Cor-Schwarz}
Let $\xi\in \partial\mathbb B$ and $f$ be a slice regular function on $\mathbb B\cup\mathbb S_{\xi}$ such that $f(\mathbb B)\subseteq\mathbb B$, $f(0)=0$ and $f(\xi)=\xi$. Then
$$f'(\xi)-\big[\xi,R_{\bar{\xi}}R_{\xi}f(\xi)\big]\geq\frac{2}{1+{\rm{Re}}f'(0)}.$$
Moreover, equality holds for the last  inequality if and only if $f$ is of the form
$$f(q)=q\big(1-qa\bar{\xi}\,\big)^{-\ast}\ast\big(q-a\xi\big)\bar{\xi}$$
for some constant $a\in [-1,1)$.
\end{corollary}

As indicated in \cite[Example 2]{WR},  the Lie bracket
in the preceding corollary does not vanish
and thus
$f'(\xi)$ is not necessarily a positive real number, in general.
 However, the same line of the proof of Theorem \ref{Generalized Herzig} implies simultaneously the following theorem, which provides a sharp lower bound for $|f'(\xi)|$.

\begin{theorem}\label{Schwarz in modulus}
Let $\xi\in \partial\mathbb B$ and $f$ be a slice regular function on $\mathbb B\cup\{\xi\}$ such that $f(\mathbb B)\subseteq\mathbb B$, $f(0)=0$ and $f(\xi)=\xi$. Then
 $$|f'(\xi)|\geq\frac{2}{1+{\rm{Re}}f'(0)}.$$
Moreover, equality holds for the last  inequality if and only if $f$ is of the form
$$f(q)=q\big(1-qa\bar{\xi}\,\big)^{-\ast}\ast\big(q-a\xi\big)\bar{\xi}$$
for some constant $a\in [-1,1)$.
\end{theorem}

We now conclude  this paper  with a comparison of the results proved in this section and the corresponding results for holomorphic self-mappings of the open unit disc on the complex plane. Even in the complex setting, the result obtained in  Theorem \ref{Generalized Herzig} is a new result. More precisely, for every  holomorphic function $f$ on $\mathbb D\cup\{1\}$ (Since the automorphism group of biholomorphisms  of the open unit disk $\mathbb D\subset\mathbb C$ acts bi-transitively on the boundary $\partial\mathbb D$, we can assume without loss of generality that the boundary point $\xi\in\partial\mathbb D$ under consideration is $1$)  satisfying that
$f(\mathbb D)\subseteq \mathbb D$ and $f(1)=1$, it can extend regularly and uniquely to $\mathbb B\cup\{1\}$. We denote (with a slight abuse of notation) this unique
 regular extension still by $f$ itself. Thus $f$ is a slice regular function on $\mathbb B\cup\{1\}$ such that $f(\mathbb B)\subseteq\mathbb B$ and $f(1)=1$. The assertion that $f(\mathbb B)\subseteq\mathbb B$ follows easily from a convex combination identity in \cite{RW2}. For all such $f$, our result becomes
\begin{equation}\label{sharp estimate}
f'(1)\geq\dfrac{2}
{\textrm{Re}\bigg(\dfrac{1-f(0)^2+f'(0)}{\big(1-f(0)\big)^2}\bigg)},
\end{equation}
which implies
$$f'(1)\geq\frac{2\big|1-f(0)\big|^2}{1-|f(0)|^2+|f'(0)|}.$$
These two inequalities improve the following estimate (also called Osserman's inequality) established by Osserman in \cite{Osserman}:
$$f'(1)\geq \frac{2\big(1-|f(0)|\big)^2}{1-|f(0)|^2+|f'(0)|}.$$

This new  estimate in  (\ref{sharp estimate}) for holomorphic self-mappings of the open unit disk $\mathbb D$, with  boundary regular fixed point $1$, was initially proved in \cite[Theorem 3]{FLSV} via an analytic semigroup  approach and Julia-Wolff-Carath\'{e}odory theorem  for univalent holomorphic self-mappings of $\mathbb D$, which was derived by the method of extremal length. The method presented in \cite{FLSV} can not be used to get the extremal functions for which equality holds in  (\ref{sharp estimate}).  The proof presented in this paper for the special case that $\xi=1$ is quite elementary, and has its extra advantage of getting the extremal functions.

\bigskip
\bibliographystyle{amsplain}

\begin{thebibliography}{99}

\bibitem{SABC} D. Alpay, V. Bolotnikov, F. Colombo, I. Sabadini, \textit{Self-mappings of the quaternionic unit ball: multiplier properties, the Schwarz-Pick inequality, and the Nevanlinna-Pick interpolation problem}. Indiana Univ. Math. J. \textbf{64} (2015), 151--180.

\bibitem{ACS} D. Alpay, F. Colombo, I. Sabadini, \textit{Slice Hyperholomorphic Schur Analysis}, Quaderni Dipartimento di Mathematica del Politecnico di Milno, QDD 209, (2015).

\bibitem{BS} C. Bisi, C. Stoppato, \textit{The Schwarz-Pick lemma for slice regular functions}. Indiana Univ. Math. J. \textbf{61 } (2012), 297--317.


\bibitem{BMMPR} R. B. Burckel, D. E. Marshall, D. Minda, P. Poggi-Corradini, T. J. Ransford, \textit{Area, capacity and diameter versions of Schwarz's lemma}. Conform. Geom. Dyn. \textbf{12} (2008), 133--152.

\bibitem{CGSS} F. Colombo, G. Gentili, I. Sabadini, D. C. Struppa, \textit{Extension results for slice regular functions of a quaternionic variable}. Adv. Math. \textbf{222} (2009), 1793--1808.

\bibitem{CLSS} F. Colombo, R. L\'{a}vi\v{c}ka, I. Sabadini, V. Sou\v{c}ek, \textit{The Radon transform between monogenic and generalized slice monogenic functions}. Math. Ann. \textbf{363 } (2015),   733--752.

\bibitem{Co3} F. Colombo, I. Sabadini, D. C. Struppa,\textit{ An extension theorem for slice monogenic functions and some of its consequences}. Israel J. Math. \textbf{177} (2010), 369--389.


\bibitem{Co2} F. Colombo, I. Sabadini, D. C. Struppa, \textit{Noncommutative functional calculus. Theory and applications of slice hyperholomorphic functions.} Progress in Mathematics, vol. 289. Birkh\"auser/Springer, Basel, 2011.
\bibitem{Co6} F. Colombo, I. Sabadini, D. C. Struppa, \textit{Slice monogenic functions}. Israel J. Math. \textbf{171} (2009), 385--403.


\bibitem{FLSV} A. Frolova, M. Levenshtein, D. Shoikhet, A. Vasil'ev, \textit{Boundary distortion estimates for holomorphic maps}. Complex Anal. Oper. Theory. \textbf{8} (2014), 1129--1149.


\bibitem{GSS2014} G. Gentili, S. Salamon, C. Stoppato, \textit{Twistor transforms of quaternionic functions and orthogonal complex structures}. J. Eur. Math. Soc. \textbf{16} (2014),  2323--2353.

\bibitem{Gen-Sar} G. Gentili, G. Sarfatti, \textit{Landau-Toeplitz theorems for slice regular functions over quaternions}. Pacific J. Math. \textbf{265} (2013),  381--404.

\bibitem{GSS1} G. Gentili, C. Stoppato, \textit{Zeros of regular functions and polynomials of a quaternionic variable}. Michigan. Math. J. \textbf{56} (2008), 655--667.

\bibitem{GSS} G. Gentili, C. Stoppato, D. C. Struppa, \textit{Regular functions of a quaternionic variable}. Springer Monographs in Mathematics, Springer, Berlin-Heidelberg, 2013.
\bibitem{GS1} G. Gentili, D. C. Struppa,\textit{ A new approach to Cullen-regular functions of a quaternionic variable}. C. R. Math. Acad. Sci. Paris, \textbf{342}  (2006), 741--744.
\bibitem{GS2} G. Gentili, D. C. Struppa, \textit{A new theory of regular functions of a quaternionic variable}. Adv. Math.  \textbf{216}  (2007), 279--301.

\bibitem{GS50} G. Gentili, D. C. Struppa, \textit{Regular functions on the space of Cayley numbers}. Rocky Mt. J. Math. \textbf{40} (2010), 225--241.
\bibitem{GV} G. Gentili, F. Vlacci, \textit{Rigidity for regular functions over Hamilton and Cayley numbers and a boundary Schwarz Lemma}. Indag. Math. (N.S.) \textbf{19} (2008), 535--545.
\bibitem{Ghiloni1} R. Ghiloni, A. Perotti, \textit{Slice regular functions on real alternative algebras}. Adv. Math. \textbf{226} (2011), 1662--1691.
\bibitem{Ghiloni2} R. Ghiloni, A. Perotti, \textit{Zeros of regular functions of quaternionic and octonionic variable: a division lemma and the
cam-shaft effect}. Ann. Mat. Pura Appl. \textbf{190} (2011), 539--551.

\bibitem{Ghiloni5} R. Ghiloni, A. Perotti, \textit{Volume Cauchy formulas for slice functions on real associative *-algebras}. Complex Var. Elliptic Equ. \textit{58} (2013),  1701--1714.

\bibitem{Ghiloni3} R. Ghiloni, A. Perotti, \textit{Power and spherical series over real alternative $\ast$-algebras}, Indiana Univ. Math. J. \textbf{63} (2014),  495--532.

\bibitem{Ghiloni4} R. Ghiloni, A. Perotti, \textit{Global differential equations for slice regular functions}. Math. Nachr. \textbf{287} (2014), 561--573.
\bibitem{Ghiloni7} R. Ghiloni, A. Perotti, C. Stoppato, \textit{The algebra of slice functions}. Trans. Amer. Math. Soc., in press.


\bibitem{Ghiloni6} R. Ghiloni, V. Recupero, \textit{Semigroups over real alternative *-algebras: Generation theorems and spherical sectorial operators}. Trans. Amer. Math. Soc. \textbf{368} (2016), 2645--2678.

\bibitem{GG} I. Graham, G. Kohr, \textit{Geometric function theory in one and higher dimensions.} Monographs and Textbooks in Pure and Applied Mathematics, 255. Marcel Dekker, Inc., New York, 2003.

\bibitem{Harvey} F. R. Harvey, \textit{Spinors and Calibrations}. Perspectives in Mathematics, Vol. 9, Academic Press, San Diego 1990.

\bibitem{Klimek} M. Klimek, \textit{Pluripotential Theory}. Oxford University Press, 1991.

\bibitem{Lam} T. Y. Lam, \textit{A First Course in Noncommutative Rings}. Springer, New York, 1991.

\bibitem{Minda} D. Minda, \textit{The Bloch and Marden constant}. Computational methods and function theory. Springer Berlin Heidelberg, 1990, pp. 131--142.

\bibitem{Okubo} S. Okubo, \textit{Introduction to Octonion and Other Non-Associative Algebras in Physics}. Cambridge University Press, 1995.
    
\bibitem{Osserman} R. Osserman, \textit{A sharp Schwarz inequality on the boundary}. Proc. Amer. Math. Soc. \textbf{128} (2000), 3513--3517.



\bibitem{Poukka} K. A. Poukka,  \textit{\"{U}ber die gr\"{o}{\ss}te Schwankung einer analytischen Funktion auf einer Kreisperipherie}. Arch. der Math. und Physik. \textbf{12} (1907), 251--254.

\bibitem{RW0} G. Ren, X. Wang, \textit{Carath\'{e}odory theorems for slice regular functions}. Complex Anal. Oper. Theory \textbf{9} (2015),  1229--1243.

\bibitem{RW} G. Ren, X. Wang, \textit{Slice regular composition operators}.  Complex Var. Elliptic Equ. \textbf{61} (2016),  682--711.

\bibitem{RW2} G. Ren, X. Wang, \textit{The growth and distortion  theorems for slice monogenic functions}. submitted. see also:  arXiv:1410.4369v2.

\bibitem{WR}  G. Ren, X. Wang, \textit{Julia theory for slice regular functions}. Trans. Amer. Math. Soc., in press.

\bibitem{Rudin} W. Rudin, \textit{Function Theory in the Unit Ball of $\mathbb C^n$}, Springer-Verlag, New York, 1987.

\bibitem{Sarason1} D. Sarason, \textit{Angular derivatives via Hilbert space}. Complex Variables. \textbf{10} (1988), 1--10.
\bibitem{Sarason2} D. Sarason, \textit{Sub-Hardy Hilbert Spaces in the Unit Disk}. University of Arkansas Lecture Notes in the Mathematical Sciences, vol. 10. Wiley, New York, 1994.
    
\bibitem{Schafer} R. D. Schafer, \textit{An Introduction to Nonassociative Algebras}. Academic Press, New York, 1966.
    
\bibitem{Serodio} R.  Ser\^{o}dio, \textit{On octonionic polynomials}. Adv. Appl. Clifford Alg. \textbf{17} (2007), 245--258.

\bibitem{Stop1} C. Stoppato, \textit{Poles of regular quaternionic functions}. Complex Var. Elliptic Equ. \textbf{54} (2009), 1001--1018.
\bibitem{Stop4}  C. Stoppato, \textit{Regular Moebius transformations of the space of quaternions}. Ann. Global Anal. Geom. \textbf{39} (2011), 387--401.
\bibitem{Stop2} C. Stoppato, \textit{Singularities of slice regular functions}. Math. Nachr. \textbf{285} (2012), 1274--1293.
\bibitem{Stop3} C. Stoppato, \textit{A new series expansion for slice regular functions}. Adv. Math. \textbf{231} (2012), 1401--1416.




\end{thebibliography}

\end{document}